\documentclass{article}
\usepackage{enumerate}
\usepackage{amsthm}
\usepackage{amsmath}
\usepackage{amsfonts}
\usepackage{amssymb}
\usepackage{amscd}
\usepackage{latexsym}
\usepackage{lscape}
\usepackage{pdflscape}

\usepackage{stmaryrd}

\usepackage{cmll}
\usepackage[table]{xcolor}
\usepackage{comment}

\usepackage{mathdots}

\usepackage{amscd}
\usepackage{amssymb}
\usepackage{amsmath}
\usepackage{mathptmx}
\usepackage{mathrsfs}
\usepackage{color}
\usepackage{xspace}
\usepackage{bussproofs}
\EnableBpAbbreviations

\usepackage{tikz}

\usepackage{centernot}

\usepackage{mathtools}

\usepackage{relsize}

\usepackage{multirow}

\usepackage{multicol}

\usepackage{array}
\newcolumntype{C}[1]{>{\centering\arraybackslash}p{#1}}
\newcolumntype{L}[1]{>{\arraybackslash}p{#1}}

\usepackage{hyperref}

\newtheorem{theorem}{Theorem}[section]

\newtheorem{lemma}[theorem]{Lemma}
\newtheorem{proposition}[theorem]{Proposition}
\newtheorem{definition}[theorem]{Definition}
\newtheorem{remark}[theorem]{Remark}

\usepackage[left=3cm,top=3cm,right=3cm,bottom=2cm]{geometry}
\newcommand{\Prop}{\mathsf{Prop}}
\newcommand{\X}{\mathbb{X}}

\newcommand{\val}[1]{[\![{#1}]\!]}
\newcommand{\descr}[1]{(\![{#1}]\!)}
\newcommand{\SI}{\mathsf{SD}}
\newcommand{\PP}{\mathsf{PP}}
\newcommand{\mathsc}[1]{{\normalfont\textsc{#1}}}
\renewcommand{\phi}{\varphi}

\title{Modelling socio-political competition}

\author{Willem Conradie$^{a}$, Alessandra Palmigiano$^{b,c}$, Claudette Robinson$^{c}$,\\ 
	Apostolos Tzimoulis$^{b}$, Nachoem Wijnberg$^{d,e}$\\ 
$^a$School of Mathematics, University of the Witwatersrand \\
$^{b}$School of Business and Economics, Vrije Universiteit, Amsterdam\\
$^c$Department of Pure and Applied Mathematics, University of Johannesburg\\
$^d$ Amsterdam Business School, University of Amsterdam\\
$^e$ College of Business and Economics,  University of Johannesburg
} 
\begin{document}
\maketitle
\begin{abstract}
This paper continues the investigation of the logic of competing theories (be they scientific, social, political etc.) initiated in \cite{eusflat}. We introduce a many-valued, multi-type modal language which we endow with relational semantics based on enriched reflexive graphs, inspired by Ploščica's representation of general lattices. We axiomatize the resulting many-valued, non-distributive modal logic of these structures and prove a completeness theorem. We illustrate the application of this logic through a case study in which we model competition among interacting political promises and social demands within an arena of political parties social groups.
\end{abstract}

%

%
\maketitle
\section{Introduction}
This paper is a continuation of the investigation into competing theories started in \cite{eusflat}.  Its technical contributions are rooted in  the generalized Sahlqvist canonicity and correspondence for {\em normal lattice-based logics} \cite{CoPa:non-dist,conradie2016constructive}, i.e.~nonclassical propositional logics for which the distributive laws between $\wedge$ and $\vee$ do not need to hold. Via algebraic and duality-theoretic techniques, these logics, and non-distributive normal modal logics in particular, have been endowed with complete relational semantics  based on {\em formal contexts}  \cite{ganter2012formal} 
and {\em reflexive graphs} \cite{conradie2015relational, graph-based-wollic}.  These semantic structures have a well developed theory, both  algebraic and proof-theoretic \cite{greco2018algebraic, craig-gouveia-haviar, craig-priestley} and model-theoretic \cite{conradie2018goldblatt}, and have facilitated new insights on possible interpretations and use of lattice-based modal logics. 

In particular, via formal context semantics, in \cite{conradie2016categories}, the basic non-distributive modal logic and some of its axiomatic extensions are interpreted as  {\em epistemic logics of categories and concepts}, and in \cite{Tarkpaper}, the corresponding `common knowledge'-type construction is used to give an epistemic-logical formalization of the notion of {\em prototype} of a category; in \cite{roughconcepts,ICLA2019paper}, formal context semantics for non-distributive modal logic is proposed as an encompassing framework for the integration of rough set theory \cite{pawlak} and formal concept analysis \cite{ganter2012formal}, and in this context, the basic non-distributive modal logic is interpreted as  the logic of {\em rough concepts};    via  graph-based semantics, in \cite{graph-based-wollic}, the same logic is interpreted as the logic of {\em informational entropy}, i.e.~an inherent boundary to knowability due e.g.~to perceptual, theoretical, evidential or linguistic limits, and in \cite{eusflat},  {\em many-valued} graph-based semantics is introduced for non-distributive normal modal logic, and its potential is explored as a formal framework for modelling {\em competing theories} in the empirical sciences. 

Both in the crisp and in the many-valued setting, in the graphs $(Z, E)$ on which the relational structures are based, the relation $E$  is interpreted as an {\em indiscernibility} relation, which makes the present approach similar to that of approximation spaces in rough set theory \cite{pawlak}. However, the key difference is that,  rather than generating modal operators which associate any subset of $Z$ with its definable $E$-approximations,  $E$ generates a complete lattice in which the distributivity laws do not need to hold. This lattice is defined as the concept lattice of the formal context $(Z, Z, E^c)$ arising from the graph $(Z, E)$. In the approach proposed in \cite{graph-based-wollic, eusflat} and followed in the present paper, concepts are not understood as definable approximations of predicates, but rather they represent `all there is to know',  i.e.~the theoretical horizon to knowability, given the inherent boundary  encoded into $E$. Interestingly, $E$ is required to be reflexive but in general neither transitive nor symmetric, which is in line with what observed in the the literature in psychology (cf.~\cite{tversky1977features,nosofsky1991stimulus}) and business science \cite{hannan}.

In this paper, we start exploring a semantic setting for non-distributive modal logics that is not only {\em many-valued}, as  the setting of \cite{eusflat} is, but unlike \cite{eusflat} is also {\em multi-type}. The main motivation and starting point  of the present contribution is to introduce a formal environment in which to analyse the similarities between the competition among {\em political theories} (both in their institutional incarnations as political parties, and in their social incarnations as social blocks or groups) and the competition between scientific theories as treated in \cite{eusflat}.

In \cite{eusflat}, scientific theories are identified with the sets of their relevant variables (e.g.~mass, speed, position are relevant variables for gravitation theory); hypotheses formulated in the background of a given theory $X$ establish connections between variables  in $X$ and are captured as formulas which can be tested  (i.e.~evaluated) on different databases (i.e.~states of the domain $Z$ of a graph-based model), with a greater or lesser degree of {\em confidence} in the outcome of the test (captured in the truth-value in the many-valued semantics). Since databases themselves are built according to a given theory (``observations are theory-laden''), the degree of confidence in the outcome of tests is formulated in terms of how compatible the background theory of the given hypothesis is with the theory according to which the given database has been built. Theories compete in the arena of databases by their key hypotheses being tested  on different databases. Then the criteria establishing whether  theory $X$ outcompetes theory $Y$ need to assign different weights to the performances of hypotheses  on databases that have high compatibility with the theories to which each hypothesis pertains, and to the performances of the same hypotheses on databases with low compatibility. In the present paper, we propose the following analogies:
\begin{center}
\begin{tabular}{r c l}
Scientific theories &$\rightsquigarrow$& Socio-political theories\\
Variables &$\rightsquigarrow$& Issues\footnote{E.g.~for a socialist theory, distribution of wealth, access to education, progressive taxation are relevant issues.}\\
Hypotheses &$\rightsquigarrow$&  Promises / Demands \\
\end{tabular}
\end{center}
The main difference between the competition of scientific theories outlined above and that of socio-political theories is that competition among  the latter plays out not on a single arena but on at least  {\em two arenas} simultaneously: that is, political parties (incarnating socio-political theories) compete with each other by testing how well their promises (phrased in terms of issues) score on different social groups, while at the same time, social groups (also incarnating socio-political theories) compete with each other by testing how well their demands score on political parties. 
The double-sidedness of this situation calls for a {\em multi-type} formal framework, both in respect to the language and the models. However, there is another interesting similarity between the socio-political case and the scientific case: as discussed above, the fact that databases are  theory-laden results in different degrees of confidence in the outcomes of tests of different hypotheses, depending on the degree of compatibility between their underlying theories; likewise, the fact that each social group has an underlying theory (captured by the set of issues which are relevant to that social group)  results in different degrees of confidence when the promises of different political parties are tested on different social groups, which again depends on the degree of compatibility between their underlying theories. Conversely and symmetrically, the fact that each political party has an underlying theory  results in different degrees of confidence when the demands of different social groups are tested on different political parties, which again depends on the degree of compatibility between their underlying theories.

\section{Preliminaries}
This section collects and modifies material from \cite[Section 2.1]{graph-based-wollic}, \cite[Section 7.2]{roughconcepts}, and \cite[Section 3]{eusflat}.
\subsection{Multi-type  nondistributive modal logic}
\label{sec:logics}
Let $\Prop$ be a (countable or finite)  set of proposition variables. The language $\mathcal{L}_{\mathsc{MT}}$ of the  {\em multi-type  nondistributive modal logic} has terms of types $\SI, \PP$ defined as follows:
\[ \SI\ni \sigma ::= \bot \mid \top \mid p \mid  \sigma \wedge \sigma\mid \sigma \vee \sigma \mid  \Diamond\pi,\]
\[ \PP\ni \pi ::= \bot \mid \top \mid p \mid  \pi \wedge \pi\mid \pi \vee \pi \mid  \lozenge\sigma,\]
where $p\in \Prop$. 
Intuitively, we create two copies of the same language, one in which formulas are intended as {\em social demands} $\sigma$ and one as {\em political promises} $\pi$. The two types are connected via heterogeneous modal operators, transforming social demands into political promises and vice versa. The {\em term-algebra} of this language is an example of {\em heterogeneous algebra}, a notion introduced by Birkhoff and Lipson \cite{birkhoff1970heterogeneous} naturally extending notions and results from universal algebra to a context in which algebras have more than one domain and operations can be defined not only within one and the same domain, but also between different domains.
\begin{definition}
\label{def:hetalgebra}
A {\em normal heterogeneous} $\mathcal{L}_{\mathsc{MT}}$-{\em algebra} is a tuple $(\mathbb{L}_S, \mathbb{L}_P, \lozenge, \Diamond)$ such that $\mathbb{L}_S$ and $\mathbb{L}_P$ are lattices (intended to interpret formulas of type $\SI$ and $\PP$, respectively), and $\lozenge: \mathbb{L}_S\to \mathbb{L}_P$ and $\Diamond: \mathbb{L}_P\to \mathbb{L}_S$ are normal (i.e.~$\bot$-preserving and $\vee$-preserving) modal operators. 
\end{definition}
The {\em basic multi-type normal} $\mathcal{L}_{\mathsc{MT}}$-{\em logic} is a set $\mathbf{L}$ of {\em type-uniform} $\mathcal{L}_{\mathsc{MT}}$-sequents $\phi\vdash\psi$  (i.e.~sequents with $\phi,\psi\in\SI$ or $\phi,\psi\in\PP$), containing the following axioms:
{\small{
		\begin{align*}
			&p\vdash p, && \bot\vdash p, && p\vdash \top, & &  &\\
			&p\vdash p\vee q, && q\vdash p\vee q, && p\wedge q\vdash p, && p\wedge q\vdash q, &\\
			& \Diamond \bot\vdash \bot, &&
			\Diamond ( \pi_1 \vee \pi_2) \vdash  \Diamond \pi_1 \vee \Diamond  \pi_2 &\\
			& \lozenge \bot\vdash \bot, &&
			 \lozenge ( \sigma_1 \vee \sigma_2)\vdash\lozenge \sigma_1 \vee \lozenge  \sigma_2  &
		\end{align*}
}}
%
%
and closed under the following inference rules:
{\small{
		\begin{displaymath}
		\frac{\phi\vdash \chi\quad \chi\vdash \psi}{\phi\vdash \psi}
		\quad
		\frac{\phi\vdash \psi}{\phi\left(\chi/p\right)\vdash\psi\left(\chi/p\right)}
		\quad
		\frac{\chi\vdash\phi\quad \chi\vdash\psi}{\chi\vdash \phi\wedge\psi}
		\quad
		\frac{\phi\vdash\chi\quad \psi\vdash\chi}{\phi\vee\psi\vdash\chi}
		\end{displaymath}
		\begin{displaymath}
		\frac{\pi_1\vdash\pi_2}{\Diamond \pi_1\vdash \Diamond \pi_2}
		\quad
		\frac{\sigma_1\vdash\sigma_2}{\lozenge \sigma_1\vdash \lozenge \sigma_2}
		\end{displaymath}
}}
An {\em $\mathcal{L}_{\mathsc{MT}}$-logic} is any  extension of $\mathbf{L}$  with type-uniform $\mathcal{L}_{\mathsc{MT}}$-sequents $\phi\vdash\psi$.
The next proposition can be shown via a routine Lindenbaum Tarski argument.
\begin{proposition}
\label{prop:completeness}
The basic logic $\mathbf{L}$ is sound and complete w.r.t.~the class of heterogeneous $\mathcal{L}_{\mathsc{MT}}$-algebras. 
\end{proposition}

\subsection{Many-valued enriched formal contexts}
Throughout this paper, we let $\mathbf{A} = (D, 1, 0, \vee, \wedge, \otimes, \to)$ denote an arbitrary but fixed complete
 frame-distributive and dually frame-distributive, commutative and associative residuated lattice(understood as the algebra of truth-values). 
For every set $W$, an $\mathbf{A}$-{\em valued subset}  (or $\mathbf{A}$-{\em subset}) of $W$ is a map $u: W\to \mathbf{A}$.  We let $\mathbf{A}^W$ denote the set of all $\mathbf{A}$-subsets. Clearly, $\mathbf{A}^W$ inherits the algebraic structure of $\mathbf{A}$ by defining the operations and the order pointwise. The $\mathbf{A}$-{\em subsethood} relation between elements of $\mathbf{A}^W$ is the map $S_W:\mathbf{A}^W\times \mathbf{A}^W\to \mathbf{A}$ defined as $S_W(f, g) :=\bigwedge_{z\in W }(f(z)\rightarrow g(z)) $. For every $\alpha\in \mathbf{A}$, 
let $\{\alpha/ w\}: W\to \mathbf{A}$ be defined by $v\mapsto \alpha$ if $v = w$ and $v\mapsto \bot^{\mathbf{A}}$ if $v\neq w$. Then, for every $f\in \mathbf{A}^W$,
\begin{equation}\label{eq:MV:join:generators}
f = \bigvee_{w\in W}\{f(w)/ w\}.
\end{equation}
 When $u, v: W\to \mathbf{A}$ and $u\leq v$ w.r.t.~the pointwise order, we write $u\subseteq v$. An $\mathbf{A}$-{\em valued relation} (or $\mathbf{A}$-{\em relation}) is a map $R: U \times W \rightarrow \mathbf{A}$. Two-valued relations can be regarded as  $\mathbf{A}$-relations. In particular for any set $Z$, we let $\Delta_Z: Z\times Z\to \mathbf{A}$ be defined by $\Delta_Z(z, z') = \top$ if $z = z'$ and $\Delta_Z(z, z') = \bot$ if $z\neq z'$. An $\mathbf{A}$-relation $R: Z\times Z\to \mathbf{A}$ is {\em reflexive} if $\Delta_Z \subseteq R$.
Any $\mathbf{A}$-valued relation $R: U \times W \rightarrow \mathbf{A}$ induces  maps $R^{(0)}[-] : \mathbf{A}^W \rightarrow \mathbf{A}^U$ and $R^{(1)}[-] : \mathbf{A}^U \rightarrow \mathbf{A}^W$ defined as follows: for every $f: U \to \mathbf{A}$ and every $u: W \to \mathbf{A}$,
\begin{center}
	\begin{tabular}{r l }
		$R^{(1)}[f]:$ & $ W\to \mathbf{A}$\\
		
		& $ x\mapsto \bigwedge_{a\in U}(f(a)\rightarrow R(a, x))$\\
		&\\	
		$R^{(0)}[u]: $ & $U\to \mathbf{A} $\\
		& $a\mapsto \bigwedge_{x\in W}(u(x)\rightarrow R(a, x))$\\
	\end{tabular}
\end{center}

A {\em formal}  $\mathbf{A}$-{\em context}\footnote{ In the crisp setting, a {\em formal context} \cite{ganter2012formal}, or {\em polarity},  is a structure $\mathbb{P} = (A, X, I)$ such that $A$ and $X$ are sets, and $I\subseteq A\times X$ is a binary relation. Every such $\mathbb{P}$ induces maps $(\cdot)^\uparrow: \mathcal{P}(A)\to \mathcal{P}(X)$ and $(\cdot)^\downarrow: \mathcal{P}(X)\to \mathcal{P}(A)$, respectively defined by the assignments $B^\uparrow: = I^{(1)}[B]$ and $Y^\downarrow: = I^{(0)}[Y]$. A {\em formal concept} of $\mathbb{P}$ is a pair 
	$c = (\val{c}, \descr{c})$ such that $\val{c}\subseteq A$, $\descr{c}\subseteq X$, and $\val{c}^{\uparrow} = \descr{c}$ and $\descr{c}^{\downarrow} = \val{c}$.   The set $L(\mathbb{P})$  of the formal concepts of $\mathbb{P}$ can be partially ordered as follows: for any $c, d\in L(\mathbb{P})$, \[c\leq d\quad \mbox{ iff }\quad \val{c}\subseteq \val{d} \quad \mbox{ iff }\quad \descr{d}\subseteq \descr{c}.\]
	With this order, $L(\mathbb{P})$ is a complete lattice, the {\em concept lattice} $\mathbb{P}^+$ of $\mathbb{P}$. Any complete lattice $\mathbb{L}$ is isomorphic to the concept lattice $\mathbb{P}^+$ of some polarity $\mathbb{P}$.} or $\mathbf{A}$-{\em polarity} (cf.~\cite{belohlavek}) is a structure $\mathbb{P} = (A, X, I)$ such that $A$ and $X$ are sets and $I: A\times X\to \mathbf{A}$. Any formal $\mathbf{A}$-context induces  maps $(\cdot)^{\uparrow}: \mathbf{A}^A\to \mathbf{A}^X$ and $(\cdot)^{\downarrow}: \mathbf{A}^X\to \mathbf{A}^A$ given by $(\cdot)^{\uparrow} = I^{(1)}[\cdot]$ and $(\cdot)^{\downarrow} = I^{(0)}[\cdot]$. 
%
These maps are such that, for every $f\in \mathbf{A}^A$ and every $u\in \mathbf{A}^X$,
\[S_A(f, u^{\downarrow}) = S_X(u, f^{\uparrow}),\]
that is, the pair of maps $(\cdot)^{\uparrow}$ and $(\cdot)^{\downarrow}$ form an $\mathbf{A}$-{\em Galois connection}.
In \cite[Lemma 5]{belohlavek}, it is shown that every  $\mathbf{A}$-Galois connection arises from some formal  $\mathbf{A}$-context. A {\em formal}  $\mathbf{A}$-{\em concept} of $\mathbb{P}$ is a pair $(f, u)\in \mathbf{A}^A\times \mathbf{A}^X$ such that $f^{\uparrow} = u$ and $u^{\downarrow} = f$. It follows immediately from this definition that if $(f, u)$ is a formal $\mathbf{A}$-concept, then $f^{\uparrow \downarrow} = f$ and $u^{\downarrow\uparrow} = u$, that is, $f$ and $u$ are {\em stable}. The set of formal $\mathbf{A}$-concepts can be partially ordered as follows:
\[(f, u)\leq (g, v)\quad \mbox{ iff }\quad f\subseteq g \quad \mbox{ iff }\quad v\subseteq u. \]
Ordered in this way, the set of the formal  $\mathbf{A}$-concepts of $\mathbb{P}$ is a complete lattice, which we denote $\mathbb{P}^+$. 

An {\em enriched formal $\mathbf{A}$-context} (cf.~\cite[Section 7.2]{roughconcepts}) is a structure  $\mathbb{F} = (\mathbb{P}, R_\Box, R_\Diamond)$ such that $\mathbb{P} = (A, X, I)$ is a formal  $\mathbf{A}$-context and $R_\Box: A\times X\to \mathbf{A}$ and $R_\Diamond: X\times A\to \mathbf{A}$ are $I$-{\em compatible}, i.e.~$R_{\Box}^{(0)}[\{\alpha / x\}]$, $R_{\Box}^{(1)}[\{\alpha / a\}]$,  $R_{\Diamond}^{(0)}[\{\alpha / a\}]$ and $R_{\Diamond}^{(1)}[\{\alpha / x\}]$ are stable for every $\alpha \in \mathbf{A}$, $a \in A$ and $x \in X$.   
The {\em complex algebra} of an  enriched formal $\mathbf{A}$-context $\mathbb{F} = (\mathbb{P}, R_\Box, R_\Diamond)$ is the algebra $\mathbb{F}^{+} = (\mathbb{P}^{+}, [R_{\Box}], \langle R_{\Diamond} \rangle )$ where $[R_{\Box}], \langle R_{\Diamond} \rangle : \mathbb{P}^{+} \to \mathbb{P}^{+}$ are defined by the following assignments: for every $c = (\val{c}, \descr{c}) \in \mathbb{P}^{+}$, 
\begin{center}
	\begin{tabular}{l cl}
		$[R_{\Box}]c$ & $ =$&$  (R_{\Box}^{(0)}[\descr{c}], (R_{\Box}^{(0)}[\descr{c}])^{\uparrow})$\\
		$ \langle R_{\Diamond} \rangle c $ & $ =$&$  ((R_{\Diamond}^{(0)}[\val{c}])^{\downarrow}, R_{\Diamond}^{(0)}[\val{c}])$.\\
	\end{tabular}
\end{center}

\begin{lemma}\label{prop:fplus}
	(cf.~\cite[Lemma 15]{roughconcepts}) If $\mathbb{F} = (\mathbb{X}, R_{\Box}, R_{\Diamond})$	is an enriched formal   $\mathbf{A}$-context, $\mathbb{F}^+ = (\mathbb{X}^+, [R_{\Box}], \langle R_{\Diamond}\rangle)$ is a complete  normal lattice expansion such that $[R_\Box]$ is completely meet-preserving and $\langle R_\Diamond\rangle$ is completely join-preserving.
\end{lemma}

\subsection{Many-valued  graphs}
A reflexive   $\mathbf{A}$-{\em graph} is a structure $\X = (Z, E)$ such that $Z$ is a nonempty set and $E$ is a {\em reflexive $\mathbf{A}$-relation}, i.e.~$E:Z\times Z\to \mathbf{A}$ and $R(z, z) = \top^\mathbf{A}$ for every $z\in Z$.  From now on, we will assume that all $\mathbf{A}$-graphs we consider are reflexive even when we drop the adjective.  

In what follows, for any set $S$ we let $S_A := \mathbf{A}\times S$ and $S_X :=  S$.
Any  $R: S\times W\to \mathbf{A}$  admits the following liftings:
\begin{center}
	\begin{tabular}{r l }
$I_R:$ & $ S_A\times W_X \to \mathbf{A}$\\
		
		& $ ((\alpha, s), w)\mapsto R(s, w)\rightarrow \alpha$\\ 
	&\\	
$J_R:$ & $ S_X\times W_A \to \mathbf{A}$\\
		
		& $ (s, (\alpha, w))\mapsto R(s, w)\rightarrow \alpha$\\ 
	\end{tabular}
\end{center}
Applying these constructions to $\mathbf{A}$-graphs we get:
\begin{definition}
\label{def:A-lifting of a graph}
For any reflexive $\mathbf{A}$-graph $\X = (Z, E)$,  the formal $\mathbf{A}$-context associated with $\X$ is
\[\mathbb{P_X}: = (Z_A, Z_X, I_E),\]
where $Z_A := \mathbf{A}\times Z$ and $Z_X :=  Z$, and  $I_E: Z_A\times Z_X\to \mathbf{A}$ is defined by $I_E((\alpha, z), z') = E(z, z')\rightarrow \alpha$. We let $\X^+: = \mathbb{P_X}^+$.\end{definition}

For all $f: \mathbf{A}\times Z\to \mathbf{A}$, and $u: Z\to \mathbf{A}$, we let
\begin{center}
	\begin{tabular}{ll}
$u^{[0]} = E^{[0]}[u]:$ & $\mathbf{A}\times Z\to \mathbf{A}$\\
& $(\alpha, z)\mapsto I_E^{(0)}[u](\alpha, z) = u^{\downarrow}(\alpha, z)$\\
&\\
$f^{[1]} = E^{[1]}[f]:$ & $Z\to \mathbf{A}$\\
& $z\mapsto I_E^{(1)}[f](z) = f^{\uparrow}(z)$\\
\end{tabular}
	\end{center}
where the maps\footnote{\label{footnote: abbreviations} We will abbreviate $E^{[0]}[u]$ and $E^{[1]}[f]$ as $u^{[0]}$ and $f^{[1]}$, respectively, for each $u, f$ as above, and write $u^{[01]}$ and $f^{[10]}$ for $(u^{[0]})^{[1]}$ and $(f^{[1]})^{[0]}$, respectively. Then $u^{[0]} = I_{E}^{(0)}[u] = u^{\downarrow}$ and $f^{[1]} = I_{E}^{(1)}[f] = f^{\uparrow}$, where the maps $(\cdot)^\downarrow$ and $(\cdot)^\uparrow$ are those associated with the polarity $\mathbb{P_X}$.} 
$f^{\uparrow}:  Z\to \mathbf{A}$ and $u^{\downarrow}: \mathbf{A}\times Z\to \mathbf{A}$ are respectively defined by the assignments \[ z\mapsto \bigwedge_{(\alpha, z')\in Z_A}[ f(\alpha, z')\to (E(z', z)\to\alpha)]\]  \[(\alpha, z)\mapsto \bigwedge_{z'\in Z_X}[ u(z')\to (E(z, z')\to\alpha)].\] 

Hence, for any $z\in Z$ and $\alpha\in \mathbf{A}$,
	\begin{center}
	\begin{tabular}{l}
	$E^{[0]}[u](\alpha, z): = \bigwedge_{z'\in Z_X}[u(z')\to (E(z, z') \to \alpha)]$\\
	$E^{[1]}[f](z): = \bigwedge_{(\alpha, z')\in Z_A}[f(\alpha, z')\to  (E(z', z) \to \alpha)]$.\\
\end{tabular}
	\end{center}

\section{Many-valued heterogeneous frames}
\begin{definition}\label{def:graph:based:frame:and:model}
If $\mathcal{L}_{\mathsc{MT}}$ denotes the multi-type language defined in Section \ref{sec:logics}, a {\em many-valued graph-based $\mathcal{L}_{\mathsc{MT}}$-frame} (abbreviated as {\em heterogeneous} $\mathbf{A}$-{\em frame}) is a structure  $\mathbb{G} = (\X_S, \X_P, R_\Diamond, R_\lozenge)$ such that $\X_S = (Z^S, E_S)$ and $\X_P = (Z^P, E_P)$ are reflexive  $\mathbf{A}$-graphs, and $R_\lozenge: Z^S\times Z^P\to \mathbf{A}$ and $R_\Diamond: Z^P\times Z^S\to \mathbf{A}$ satisfy the following {\em compatibility} conditions:\footnote{In what follows, we drop the indices whenever a property, notion or construction applies verbatim to both domains or when disambiguation can be achieved with other means. For instance, symbols such as $(\cdot)^{[0]}$ and $(\cdot)^{[1]}$ will never occur with indices, since the type of the argument is enough to disambiguate them.}
for any $z\in Z^S$, $z'\in Z^P$ and $\alpha, \beta\in\mathbf{A}$,
	\begin{center}
		\begin{tabular}{r c l }
		$(R_\Diamond^{[0]}[\{\beta / (\alpha, z') \}])^{[01]}$&$ \subseteq $&$ R_\Diamond^{[0]}[\{\beta /  (\alpha, z') \}] $\\
			$    (R_\Diamond^{[1]}[\{\beta /  z' \}])^{[10]} $&$ \subseteq $&$  R_\Diamond^{[1]}[\{\beta / z' \}]$\\
			$(R_\lozenge^{[0]}[\{\beta / (\alpha, z) \}])^{[01]} $&$ \subseteq $&$ R_\lozenge^{[0]}[\{\beta /  (\alpha, z) \}] $ \\
			$  (R_\lozenge^{[1]}[\{\beta /  z \}])^{[10]} $&$ \subseteq $&$  R_\lozenge^{[1]}[\{\beta /  z \}]$.\\
		\end{tabular}
	\end{center}
	
where for all $f: Z^S_A\to \mathbf{A}$ and $u: Z^P_X\to \mathbf{A}$,
\begin{center}
	\begin{tabular}{ll}
$R_{\Diamond}^{[0]}[f]:$ & $Z^P_X\to \mathbf{A}$\\
& $z\mapsto  J_{R_\Diamond}^{(0)}[f]( z)$\\
&\\	
$R_{\Diamond}^{[1]}[u]:$ & $Z^S_A\to \mathbf{A}$\\
& $(\alpha, z)\mapsto  J_{R_\Diamond}^{(1)}[u](\alpha, z)$,\\
\end{tabular}
	\end{center}
and for all $f: Z^P_A\to \mathbf{A}$ and $u: Z^S_X\to \mathbf{A}$,
\begin{center}
	\begin{tabular}{ll}
$R_{\lozenge}^{[0]}[f]:$ & $Z^S_X\to \mathbf{A}$\\
& $z\mapsto  J_{R_\lozenge}^{(0)}[f]( z)$\\
&\\	
$R_{\lozenge}^{[1]}[u]:$ & $Z^P_A\to \mathbf{A}$\\
& $(\alpha, z)\mapsto  J_{R_\lozenge}^{(1)}[u](\alpha, z)$.\\
\end{tabular}
	\end{center}
	Hence, for any $z\in Z^P$, $w\in Z^S$ and $\alpha\in \mathbf{A}$,
	\begin{center}
	\begin{tabular}{l}
$R_{\Diamond}^{[0]}[f](z): = \bigwedge_{(\alpha, z')\in Z^S_A}[f(\alpha, z')\to  (R_{\Diamond}(z, z') \to \alpha)]$\\
	$R_{\Diamond}^{[1]}[u](\alpha, w): = \bigwedge_{z'\in Z^P_X}[u(z')\to (R_{\Diamond}(z', w) \to \alpha)]$,\\
	\end{tabular}
	\end{center}	
and for any $z\in Z^S$, $w\in Z^P$ and $\alpha\in \mathbf{A}$,
	\begin{center}
	\begin{tabular}{l}
$R_{\lozenge}^{[0]}[f](z): = \bigwedge_{(\alpha, z')\in Z^P_A}[f(\alpha, z')\to  (R_{\lozenge}(z, z') \to \alpha)]$\\
	$R_{\lozenge}^{[1]}[u](\alpha, w): = \bigwedge_{z'\in Z^S_X}[u(z')\to (R_{\lozenge}(z', w) \to \alpha)]$.\\
	\end{tabular}
	\end{center}	
	
	The {\em complex algebra} of a heterogeneous $\mathbf{A}$-frame $\mathbb{G}$ as above is the heterogeneous algebra $\mathbb{G}^+ = (\mathbb{X}_S^+,\mathbb{X}_P^+, \langle R_\lozenge\rangle, \langle R_\Diamond\rangle)$,
	where $\mathbb{X}_S^+: = \mathbb{P}_{\mathbb{X}_S}^+$ and $\mathbb{X}_P^+: = \mathbb{P}_{\mathbb{X}_P}^+$ (cf.~Definition \ref{def:A-lifting of a graph}), and $\langle R_\lozenge\rangle: \mathbb{X}_P^+\to \mathbb{X}_S^+$ and $\langle R_\Diamond\rangle: \mathbb{X}_S^+\to \mathbb{X}_P^+$ are heterogeneous operations of $\mathbb{G}^+$ defined as follows: for every $c = (\val{c}, \descr{c}) \in \mathbb{X}_S^+$ and $d = (\val{d}, \descr{d}) \in \mathbb{X}_P^+$,
	\begin{center}
		\begin{tabular}{l cl}
					$ \langle R_\Diamond\rangle c$ & $ =$&$  ((R_{\Diamond}^{[0]}[\val{c}])^{[0]}, R_{\Diamond}^{[0]}[\val{c}])$\\
						$\langle R_\lozenge\rangle d$ & $ =$&$   ((R_{\lozenge}^{[0]}[\val{d}])^{[0]}, R_{\lozenge}^{[0]}[\val{d}]) $.\\

		\end{tabular}
	\end{center}
\end{definition}

With a proof analogous to that of Lemma \ref{prop:fplus}, one can readily show that
\begin{lemma}
	If $\mathbb{G} = (\X_S, \X_P, R_\Diamond, R_\lozenge)$	is a  heterogeneous $\mathbf{A}$-frame, $\mathbb{G}^+ = (\mathbb{X}_S^+,\mathbb{X}_P^+, \langle R_\lozenge\rangle, \langle R_\Diamond\rangle)$ is such that $\mathbb{X}_S^+$ and $\mathbb{X}_P^+$ are complete lattices, and   $\langle R_\lozenge\rangle$  and $\langle R_\Diamond\rangle$ are completely join-preserving.
\end{lemma}

\section{Many-valued heterogeneous models} \label{sec:interpretation}
Let $\mathcal{L}_{\mathsc{MT}}$ be the language of Section \ref{sec:logics}.  
\begin{definition}
	A {\em graph-based}  $\mathbf{A}$-{\em model} of $\mathcal{L}_{\mathsc{MT}}$  is a tuple $\mathbb{M} = (\mathbb{G}, V)$ such that $\mathbb{G} = (\X_S, \X_P, R_\Diamond, R_\lozenge)$	is a  heterogeneous $\mathbf{A}$-frame, and $V : \mathcal{L}\to \mathbb{G}^+$ is a homomorphism of heterogeneous algebras, i.e.~$V$ is a pair of homomorphisms\footnote{Notice the inversion: formulas of type $\SI$ (social demands) are evaluated (tested) on the $P$-side of the model, i.e.~on political parties, and conversely, political promises are evaluated on social groups. Hence, the complex algebra $\mathbb{G}^+ = (\mathbb{X}_S^+,\mathbb{X}_P^+, \langle R_\lozenge\rangle, \langle R_\Diamond\rangle)$ of the underlying frame of a model is a heterogeneous $\mathcal{L}_{\mathrm{MT}}$-algebra in the sense that $\mathbb{L}_P: = \mathbb{X}_S^+$ and $\mathbb{L}_S := \mathbb{X}_P^+$, and moreover $\Diamond:= \langle R_\Diamond\rangle: \mathbb{L}_P\to \mathbb{L}_S$ and $\lozenge: = \langle R_\lozenge\rangle: \mathbb{L}_S\to \mathbb{L}_P$.} $V_S: \SI\to \mathbb{X}_P^+$ and $V_P: \PP\to \mathbb{X}_S^+$. For every $\phi\in \mathcal{L}_{\mathsc{MT}}$, let $V(\phi): = (\val{\phi}, \descr{\phi})$, where $\val{\phi}: \mathbf{A}\times Z\to \mathbf{A}$ and $\descr{\phi}: Z\to \mathbf{A}$, with $Z$ being the domain of the appropriate type, are s.t.~$\val{\phi}^{[1]} = \descr{\phi}$ and $\descr{\phi}^{[0]} = \val{\phi}$. Hence:
	%
	\begin{center}
		\begin{tabular}{r c l}
			$V(p)$ & = & $(\val{p}, \descr{p})$\\
			$V(\top)$ & = & $(1^{\mathbf{A}^{Z_A}}, (1^{\mathbf{A}^{Z_A}})^{[1]})$\\
			$V(\bot)$ & = & $((1^{\mathbf{A}^{Z_X}})^{[0]}, 1^{\mathbf{A}^{Z_X}})$\\
			$V(\phi\wedge \psi)$ & = & $(\val{\phi}\wedge\val{\psi}, (\val{\phi}\wedge\val{\psi})^{[1]})$\\
			$V(\phi\vee \psi)$ & = & $((\descr{\phi}\wedge\descr{\psi})^{[0]}, \descr{\phi}\wedge\descr{\psi})$\\
			$V(\lozenge\sigma)$ & = & $((R_{\lozenge}^{[0]}[\val{\sigma}])^{[0]}, R^{[0]}_\lozenge[\val{\sigma}])$.\\
			$V(\Diamond\pi)$ & = & $((R^{[0]}_\Diamond[\val{\pi}])^{[0]}, R^{[0]}_\Diamond[\val{\pi}])$.\\
		\end{tabular}
	\end{center}
	Valuations induce  $\alpha$-{\em support relations} between value-state pairs and formulas of the appropriate type for each $\alpha\in \mathbf{A}$ (in symbols: $\mathbb{M}, (\beta, z)\Vdash^\alpha \phi$), and $\alpha$-{\em refutation relations} between states of models and formulas for each $\alpha\in \mathbf{A}$ (in symbols: $\mathbb{M},  z\succ^\alpha \phi$) such that for every $\phi\in \mathcal{L}_{\mathsc{MT}}$, all $z\in Z$ and all $\beta\in \mathbf{A}$,
\begin{center}
	\begin{tabular}{r c l}
$\mathbb{M}, (\beta, z)\Vdash^\alpha \phi$ & iff  &$ \alpha\leq \val{\phi}(\beta, z),$\\
$\mathbb{M},z \succ^\alpha \phi $ & iff  &$ \alpha\leq \descr{\phi}( z).$\\
\end{tabular}
\end{center}

This can be equivalently expressed as follows:
{\small{
\begin{center}
\begin{tabular}{l c l}
$\mathbb{M}, (\beta, z)\Vdash^\alpha p$ & iff & $\alpha\leq \val{p}(\beta, z)$;\\
$\mathbb{M}, (\beta, z)\Vdash^\alpha \top$ & iff & $\alpha\leq 1^{\mathbf{A}^{Z_A}}(\beta, z)$ i.e.~always;\\
$\mathbb{M}, (\beta, z)\Vdash^\alpha \bot$ & iff & $\alpha\leq (1^{\mathbf{A}^{Z_X}})^{[0]} (\beta, z)$\\
&& $ =\bigwedge_{z'\in Z_X}[1^{\mathbf{A}^{Z_X}}(z')\to (E(z, z') \to \beta)]$\\
&& $ =\bigwedge_{z'\in Z_X}[E(z, z') \to \beta]$\\
&& $ =\beta$;\\
$\mathbb{M}, (\beta, z)\Vdash^\alpha \phi\wedge \psi$ & iff & $\mathbb{M}, (\beta, z)\Vdash^\alpha \phi\, $ and $\, \mathbb{M}, (\beta, z)\Vdash^\alpha \psi$;\\
$\mathbb{M}, (\beta, z)\Vdash^\alpha \phi\vee \psi$ & iff & $\alpha\leq (\descr{\phi}\wedge\descr{\psi})^{[0]}(\beta, z)$\\
&& $ = \bigwedge_{z'\in Z_X}[(\descr{\phi}\wedge\descr{\psi})(z')\to (E(z, z')$\\ 
&&$\to \beta)]$;\\
$\mathbb{M}, (\beta, z)\Vdash^\alpha \lozenge \sigma$ & iff & $\alpha\leq ((R^{[0]}_\lozenge[\val{\sigma}])^{[0]})(\beta, z)$\\
&&$ = \bigwedge_{z'\in Z_X^S}[R^{[0]}_\lozenge[\val{\sigma}](z')\to (E_S(z, z') \to \beta)]$;\\

$\mathbb{M}, (\beta, z)\Vdash^\alpha \Diamond \pi$ & iff & $\alpha\leq ((R^{[0]}_\Diamond[\val{\pi}])^{[0]})(\beta, z)$\\
&&$ = \bigwedge_{z'\in Z_X^P}[R^{[0]}_\Diamond[\val{\pi}](z')\to (E_P(z, z') \to \beta)]$;\\
\end{tabular}
\end{center}
\begin{center}
\begin{tabular}{l c l}
$\mathbb{M}, z\succ^\alpha p$ & iff & $\alpha\leq \descr{p}( z)$;\\
$\mathbb{M}, z\succ^\alpha \bot$ & iff & $\alpha\leq 1^{\mathbf{A}^{Z_X}}( z)$ i.e.~always;\\
$\mathbb{M},  z\succ^\alpha \top$ & iff & $\alpha\leq (1^{\mathbf{A}^{Z_A}})^{[1]} ( z) $\\
&&$ = \bigwedge_{(\beta, z')\in Z_A}[1(\beta, z')\to  (E(z', z) \to \beta)]$\\
&&$ = \bigwedge_{(\beta, z')\in Z_A}[E(z', z) \to \beta]$\\
&& $ =\beta$;\\
$\mathbb{M},  z\succ^\alpha \phi\vee \psi$ & iff & $\mathbb{M},  z\succ^\alpha \phi$ and $\mathbb{M},  z\succ^\alpha \psi$;\\
$\mathbb{M}, z\succ^\alpha \phi\wedge \psi$ & iff & $\alpha\leq (\val{\phi}\wedge\val{\psi})^{[1]}(z)$\\
&&$ = \bigwedge_{(\beta, z')\in Z_A}[(\val{\phi}\wedge\val{\psi})(\beta, z')\to  (E(z', z)$\\
&&$\to \beta)]$;\\

$\mathbb{M}, z\succ^\alpha \lozenge \sigma$ & iff & $\alpha\leq (R^{[0]}_\lozenge[\val{\sigma}])(z)$\\
&& $ = \bigwedge_{ (\beta, z')\in Z_A^P}[\val{\sigma}(\beta, z')\to (R_{\lozenge}(z, z')\to \beta)]$;\\

$\mathbb{M}, z\succ^\alpha \Diamond \pi$ & iff & $\alpha\leq (R^{[0]}_\Diamond[\val{\pi}])(z)$\\
&& $ = \bigwedge_{ (\beta, z')\in Z_A^S}[\val{\pi}(\beta, z')\to (R_{\Diamond}(z, z')\to\beta)]$.\\
\end{tabular}
\end{center}
}}
\end{definition}
\begin{definition}\label{def:Sequent:True:In:Model}
	A type-uniform sequent $\phi \vdash \psi$ is \emph{true} in a model $\mathbb{M} = (\mathbb{G}, V)$ (notation:  $\mathbb{M} \models \phi \vdash \psi$) if  $\val{\phi}\subseteq \val{\psi}$, or equivalently, if $\descr{\psi}\subseteq \descr{\phi}$.
	A type-uniform sequent $\phi \vdash \psi$ is \emph{valid on a graph-based frame} $\mathbb{G}$ (notation:  $\mathbb{G} \models \phi \vdash \psi$) if $\phi \vdash \psi$ is true in every model $\mathbb{M} = (\mathbb{G}, V)$ based on $\mathbb{G}$.
\end{definition}
\begin{remark}\label{remark:Monotone:Vals}
As remarked in \cite{eusflat}, it is not difficult to see that for all stable valuations, if $p\in \Prop$ and $\beta, \beta'\in \mathbf{A}$ such that $\beta\leq \beta'$, then $\val{p}(\beta, z) \leq \val{p}(\beta', z)$ for every $z\in Z$, and one can readily verify that this condition extends compositionally to every $\phi\in \mathcal{L}$. 
\end{remark}

\section{Case study: the socio-political arena}\label{sec:Case:Study}

Let $\mathsf{Var}$ be a nonempty set of variables  (intended to represent {\em topics} or  {\em issues}, as in e.g.~\cite{arjen}). As was done in \cite{eusflat}, for the purpose of this analysis, a socio-political theory is characterized by (and here identified with) a certain subset $X\subseteq \mathsf{Var}$ of issues which are {\em relevant} to the given theory.  The heterogeneous $\mathbf{A}$-frames considered in the present section are structures  $\mathbb{G} = (\X_S, \X_P, R_\Diamond, R_\lozenge)$ (cf.~Definition \ref{def:graph:based:frame:and:model}) such that  $\X_S: = (Z^S, E_S)$  with $Z^S: = \{z_{X_i}\mid 0\leq i\leq n\}$ where $X_i\subseteq \mathsf{Var}$ for each $1 \leq i \leq n$ and $\X_P: = (Z^P, E_P)$ with $Z^P: = \{z_{X_j}\mid 0\leq i\leq m\}$, again with $X_j\subseteq \mathsf{Var}$ for each $1 \leq j \leq m$. 
The elements of the sets $Z^S$ and $Z^P$ stand for {\em social groups} and {\em political parties}, respectively. The set of variables indexing each social group in $ Z^S$ (resp.~each political party in $Z^P$) stand for the issues considered relevant by that social group or political party. Sometimes, it can be useful to encode the positive or negative orientation of the group/party towards each relevant issue by assigning a {\em sign} ($+$ or $-$) to each element of the indexing set $X_i$ or $X_j$.\footnote{Of course,  a sign is often not enough to achieve a full disambiguation; however, for the sake of the example below, what will matter is whether a given party and social group assign the same or opposite sign to a given issue relevant to both.} In this context, $\SI$-formulas (resp.~$\PP$-formulas) can be thought of as {\em social demands} (resp.~{\em political promises}) which will be `tested' (i.e.~will be assigned truth-degrees) at  states of $Z_P$ (resp.~$Z_S$), i.e.~at political parties (resp.~social groups) in models based on these frames. Notice the inversion: $\SI$-formulas will be evaluated at $\X_P$-states, and $\PP$-formulas at $\X_S$-states. This truth value assignment of formulas at states  is then meant to act as a proxy for the support (or interest) of the given social group in the given political promise, and of the support (or interest) of the political party in the given social demand,  with higher truth values   indicating higher levels of support/interest.

The  $\mathbf{A}$-relation $E_S: Z^S\times Z^S\to \mathbf{A}$ (resp.~$E_P: Z^P\times Z^P\to \mathbf{A}$) encodes a graded notion of {\em similarity} between social groups (resp.~political parties). This  idea can be concretely implemented e.g.~by letting $E(z_{X_1}, z_{X_2})$ record the percentage of variables of $z_{X_1}$ that also occur in $z_{X_2}$, i.e by taking $E(z_{X_1}, z_{X_2}) = |X_1 \cap X_2| \div |X_1|$, possibly modulo identification of similar issues.\footnote{In this paper, we are not actually committing to a specific definition of the similarity relations, although the one above naturally arises from the present formal framework and will be employed in the case study we present.} As remarked in  \cite{eusflat}, a relation defined accordingly  will be reflexive (i.e.~$E(z, z) = 1$ for every $z\in Z$) but does not need to be symmetric or transitive; moreover, it is  not required to record the positive or negative attitudes vis-\`a-vis an issue, so as to not exclude the possibility that social groups (resp.~parties) with directly opposing views on a large percentage of issues have a high similarity  degree.

The $\mathbf{A}$-relation $R_\Diamond: Z^P\times Z^S\to \mathbf{A}$ (resp.~$R_\lozenge: Z^S\times Z^P\to \mathbf{A}$)  encodes the extent to which a political party (resp.~social group) has {\em affinity} with a social group (resp.~political party). One would expect that such a measure should be based on the extent to which the political party (resp.~social group) perceives its issues to be issues of the social group (resp.~political party). This idea can be concretely implemented e.g.~as follows: Let $X_{P} \subseteq \mathsf{Var}$ (resp.~$X_{S} \subseteq \mathsf{Var}$) be the set of all issues of political parties (resp.~social groups). Encode the extent to which a political party $z_{X_j}$ recognizes each of its issues in an issue of a social group, using a \emph{recognition function} $f_{z_{X_j}}: X_j \times X_S \to [0,1]$. We then set  $R_\Diamond (z_{X_j}, z_{X_i}) = \sum \{ f(x,y) \mid (x,y) \in X_j \times X_i \}  \div |\{ (x,y) \in X_j \times X_i \mid f(x,y) \neq 0 \}|$. Recognition functions for social groups and the resulting definition of $R_\lozenge$ are analogous. More nuanced realisations might, among other considerations, also include a weighting to account for the relative importance of issues to the political parties or groups.  Notice that we are {\em not} requiring, because it would be implausible, that $R_\Diamond (z_{X_j}, z_{X_i}) = R_\lozenge (z_{X_i}, z_{X_j})$ for all $z_{X_j}$ and $z_{X_i}$. Below, we give a more concrete illustration of this environment by means of an example loosely inspired by the British socio-political scene.  

Let $\mathsf{Var}: = \{st, o, lt, ap, ft, cr, it, fs, h, at, s\}$ be the set of issues, where the intended meaning of each variable is indicated below:
\begin{center}
	\begin{tabular}{c c l}
		$st$ && lower income {\bf t}ax on {\bf s}alaries\\
		$o$  && foreigners {\bf o}ut\\
		$lt$ && lower {\bf t}axes on income generated from {\bf l}and\\
		$ap$ && preservation of {\bf a}ristocratic {\bf p}rivileges\\
		$ft$ && lower {\bf f}inancial transactions {\bf t}ax\\
		$cl$ && harmonization of European {\bf c}orporation {\bf l}aw\\
		$it$ && progressive {\bf i}ncome {\bf t}ax\\
		$fs$ && higher tax on {\bf f}oreign {\bf s}tocks flotation on the London stock exchange\\
		$h$ && fox {\bf h}unting\\
		$at$ && lower {\bf t}ax on {\bf a}gricultural sector\\ 
		$s$ && national {\bf s}overeignty\\
		$ur$ &&reduced rights for {\bf u}nion {\bf r}epresentatives in factories\\
		$ds$ && tax {\bf d}eductions for {\bf s}avings of lower income workers\\
		$pd$ && return to the {\bf p}re-{\bf d}ecimal currency system\\
	\end{tabular}
\end{center} 
Let 
$Z_S: = \{z_{F}, z_{D}, z_{B}\}$ and $Z_P: = \{z_{L}, z_{C}, z_{X}\}$, where 
\begin{center}
	\begin{tabular}{c  l l}
		$z_F$ & {\bf F}actory workers in Manchester & $F: = \{+st, +o\}$\\
		$z_D$ &Extended family of {\bf D}uke of Westminster & $D: = \{+lt, +ap\}$\\
		$z_B$ & London City {\bf B}ankers  & $B: = \{+ft, -cl\}$\\
		\hline
		$z_L$ & {\bf L}abour party  & $L: = \{+it, +fs, -h\}$\\
		$z_C$ & {\bf C}onservative party  & $C: = \{+at, +h, +ur \}$\\
		$z_X$ & Bre{\bf x}it  party& $X: = \{+s, +ds, +pd \}$\\
	\end{tabular}
\end{center} 

To calculate the similarity between political parties, we need to compare their positions in terms of issues. However, since there is ostensibly little overlap in the issues \emph{as formulated}, we will impose an equivalence relation $\sim_{P}$ to abstract the broad \emph{kinds of issues} and use that for the comparison. Suppose the equivalence classes of $\sim_{P}$ are given by grouping income tax issues together and `heritage issues' together, while keeping other issues separate, as follows:
\[
\{\{it, at, ds \}, \{fs\}, \{h, pd \}, \{s\}, \{ur \} \}.
\]
In the same way, to calculate the similarity between social groups, we impose an equivalence relation $\sim_S$ on their issues, equating tax issues while distinguishing other issues:
\[
\{\{st, lt, ft \}, \{o\}, \{ap\}, \{cl\} \}.
\]
The equivalence class of an issue $i$ under $\sim_P$ (respectively, $\sim_S$) is denoted by $[i]_P$ (respectively, $[i]_S$). The similarity relations between $Z^S$ and $Z^P$ can take values in the  11-element {\L}ukasiewicz chain $\mathbf{A}$ with domain $\{0, 0.1, 0.2, \ldots, 1\}$, as indicated in the following diagram:

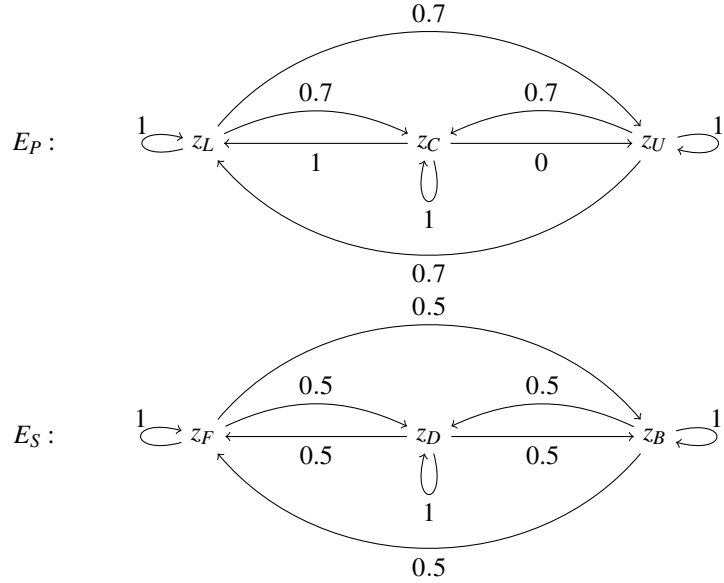
\begin{center}
	\begin{tikzpicture}[scale=1.5]
	\node (F)    at (-2,0)    {$z_F$};
	\node (D)   at (0,0)    {$z_D$};
	\node (B)    at (2,0)    {$z_B$};
	\node (L)    at (-2,2.6)    {$z_L$};
	\node (C)   at (0,2.6)    {$z_C$};
	\node (U)    at (2,2.6)    {$z_U$};
	\draw[->] (D) -- (F) node [midway,below]{0.5};
	\draw[->] (D) -- (B) node [midway,below]{0.5};
	\draw[->, bend left = 25]  (F) to node [midway,above]{$0.5$} (D);
	\draw[->, bend right = 25]  (B) to node [midway,above]{$0.5$} (D);
	\draw[->,   loop below, min distance=5mm]  (D) to node [midway,below]{$1$} (D);
	\draw[->,  loop left, min distance=5mm]  (F) to node [midway,above]{$1$} (F);
	\draw[->,  loop right, min distance=5mm]  (B) to node [midway,above]{$1$} (B);
	\draw[->,  bend left=50] (F) to node [midway,above]{$0.5$} (B);
	\draw[->,  bend left=50] (B) to node [midway,below]{$0.5$} (F);	
	
	\draw[->] (C) -- (L) node [midway,below]{1};
	\draw[->] (C) -- (U) node [midway,below]{0};
	\draw[->, bend left = 25]  (L) to node [midway,above]{$0.7$} (C);
	\draw[->, bend right = 25]  (U) to node [midway,above]{$0.7$} (C);
	\draw[->,   loop below, min distance=5mm]  (C) to node [midway,below]{$1$} (C);
	\draw[->,  loop left, min distance=5mm]  (L) to node [midway,above]{$1$} (L);
	\draw[->,  loop right, min distance=5mm]  (U) to node [midway,above]{$1$} (U);
	\draw[->,  bend left=50] (L) to node [midway,above]{$0.7$} (U);
	\draw[->,  bend left=50] (U) to node [midway,below]{$0.7$} (L);	
	
	\node (labelEP)    at (-3.5,2.6)    {$E_P:$};
	\node (labelES)    at (-3.5,0)    {$E_S:$};
	
	\end{tikzpicture}	
\end{center} 
The values of these relations are calculated according to the formula given above, with rounding as necessary. For example, $E_S(z_F, z_D) = |\{[st]_S, [o]_S\} \cap \{[lt]_S, [h]_S\}| \div |\{[st]_S, [o]_S\}| = |\{[st]_S\} |\div |\{[st]_S, [o]_S\}| = 0.5$.  In order to define the relation $R_\lozenge: Z^S\times Z^P\to \mathbf{A}$, we will use the following recognition functions:

\begin{align*}
&\begin{array}{l|llllllll}
f_F &fs &it &h &at &s &ur &ds &pd\\
\hline
st &0 &0.9 &0 &0.2 &0 &0 &0.2 &0\\
o  &0.5 &0 &0 &0 &0.4 &0 &0 &0\\
\end{array}
\quad \quad 
\begin{array}{l|llllllll}
f_D &fs &it &h &at &s &ur &ds &pd\\
\hline
lt &0 &0.2 &0 &0.6 &0 &0 &0 &0\\
ap &0 &0 &0.8 &0 &0.2 &0 &0 &0 \\
\end{array}
\\
&\begin{array}{l|llllllll}
f_B &fs &it &h &at &s &ur &ds &pd\\
\hline
ft &0 &0 &0 &0.4 &0 &0 &0.3 &0\\
cl &0.3 &0 &0 &0 &0.5 &0 &0 &0\\
\end{array}
\end{align*} 
This enables us to calculate, e.g.\ $R_{\lozenge}(z_F, z_L) = [(f_{F}(st,it) + f_{F}(o,fs)) \div |\{f_{F}(st,it),  f_{F}(o,st) \}|] = (0.9 + 0.5) \div 2 = 0.7$. The complete relation $R_{\lozenge}(z_F, z_L)$ is depicted on the following figure:

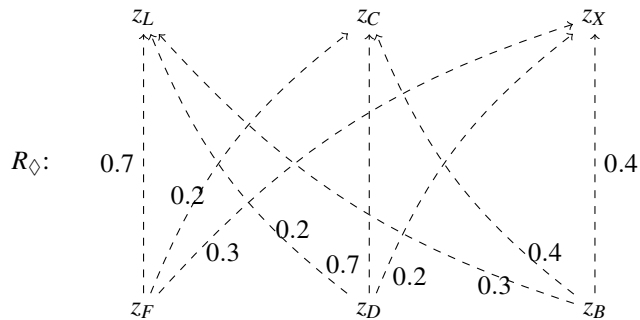
\begin{center}
	\begin{tikzpicture}[scale=1.5]
	\node (F)    at (-2,0)    {$z_F$};
	\node (D)   at (0,0)    {$z_D$};
	\node (B)    at (2,0)    {$z_B$};
	
	\node (L)    at (-2,2.6)    {$z_L$};
	\node (C)   at (0,2.6)    {$z_C$};
	\node (U)    at (2,2.6)    {$z_X$};

	\draw[->, bend left = 0, dashed]  (F) to node [midway,left]{$0.7$} (L);	
	\draw[->, bend left = 15, dashed]  (F) to node [near start,above]{$0.2$} (C);	
	\draw[->, bend left = 15, dashed]  (F) to node [pos=0.1,right]{$0.3$} (U);	
	
	\draw[->, bend left = 15, dashed]  (D) to node [pos=0.2,above]{$0.2$} (L);	
	\draw[->, bend left = 0, dashed]  (D) to node [pos=0.1,left]{$0.7$} (C);	
	\draw[->, bend left = 15, dashed]  (D) to node [pos=0.05,right]{$0.2$} (U);	
	
	\draw[->, bend left = 15, dashed]  (B) to node [pos=0.1,left]{$0.3$} (L);	
	\draw[->, bend left = 15, dashed]  (B) to node [pos=0.1,above]{$0.4$} (C);	
	\draw[->, bend left = 0, dashed]  (B) to node [midway,right]{$0.4$} (U);	
	
	\node (F)    at (-3,1.3)    {$R_{\lozenge}$:};
	
	\end{tikzpicture}	
\end{center} 
%

The recognition function of the political parties are given by:
\begin{align*}
&\begin{array}{l|llllll}
f_L &st &o &lt &ap &ft &cl\\
\hline
fs &0 &0.2 &0 &0 &0 &0.3\\
it &0.8 &0 &0.2 &0 &0 &0\\
h  &0 &0 &0 &0 &0 &0
\end{array}
\quad \quad
\begin{array}{l|llllll}
f_C &st &o &lt &ap &ft &cl\\
\hline
at &0.3 &0 &0.4 &0 &0 &0\\
h  &0 &0 &0 &0.6 &0 &0\\
ur &0 &0 &0 &0 &0 &0.4\\
\end{array}
\\
&\begin{array}{l|llllll}
f_X &st &o &lt &ap &ft &cl\\
\hline
s  &0 &0.8 &0 &0.2 &0 &0.5\\
ds &0.4 &0 &0 &0 &0.3 &0\\
pd &0 &0 &0 &0.2 &0 &0\\
\end{array}
\end{align*}
Based on these recognition functions, the relation  $R_\Diamond: Z^P\times Z^S\to \mathbf{A}$ is calculated, and given the figure below:

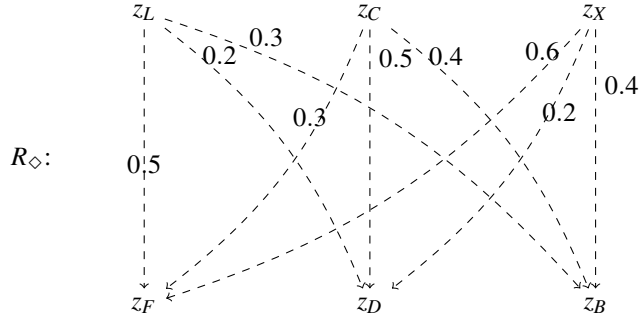
\begin{center}
	\begin{tikzpicture}[scale=1.5]
	\node (F)    at (-2,0)    {$z_F$};
	\node (D)   at (0,0)    {$z_D$};
	\node (B)    at (2,0)    {$z_B$};
	
	\node (L)    at (-2,2.6)    {$z_L$};
	\node (C)   at (0,2.6)    {$z_C$};
	\node (U)    at (2,2.6)    {$z_X$};

	\draw[->, bend left = 0, dashed]  (L) to node [pos=0.6,above]{$0.5$} (F);	
	\draw[->, bend left = 15, dashed]  (C) to node [pos=0.35,above]{$0.3$} (F);	
	\draw[->, bend left = 15, dashed]  (U) to node [pos=0.1,above]{$0.6$} (F);
	
	\draw[->, bend left = 15, dashed]  (L) to node [pos=0.2,above]{$0.2$} (D);	
	\draw[->, bend left = 0, dashed]  (C) to node [pos=0.1,right]{$0.5$} (D);	
	\draw[->, bend left = 15, dashed]  (U) to node [pos=0.2,below]{$0.2$} (D);
	
	\draw[->, bend left = 15, dashed]  (L) to node [pos=0.2,above]{$0.3$} (B);	
	\draw[->, bend left = 15, dashed]  (C) to node [pos=0.2,above]{$0.4$} (B);	
	\draw[->, bend left = 0, dashed]  (U) to node [pos=0.2,right]{$0.4$} (B);
	
	\node (F)    at (-3,1.3)    {$R_{\Diamond}$:};
	
	\end{tikzpicture}	
\end{center} 
Notice that e.g.~$z_L$ has a {\em lower} degree of affinity to $ z_F$ than $z_F$ has to $z_L$; this difference is due to the asymmetry between the way in which the Labour party recognises its issues in the issues of the factory workers and the way in which factory workers recognise theirs in those of the Labour party.  

Let $\sigma_F,\sigma_D, \sigma_B\in \SI$ and $\pi_L,\pi_C, \pi_U\in \PP$ respectively represent the following social demands and political promises:


\begin{center}
	\begin{tabular}{l c l}
		$\sigma_F(st, o)$ && lower taxes on salaries by cutting social benefits for foreigners\\
		$\sigma_D(lt, ap)$ && right for the house of lords to veto laws on grounds of national interest\\
		$\sigma_B(ft, cl)$ && UK companies are allowed to pay the salaries of their employees only through UK banks\\ 
		$\pi_L(it, fs, h)$ && tax-money used to enforce fox hunting ban \\
		$\pi_C(at, h)$ && reducing the use of tax-money for  enforcing fox hunting ban\\
		$\pi_X(s)$ && increasing national sovereignty\\
	\end{tabular}
\end{center}

Each demand (resp.~promise) is phrased in terms of some of the issues relevant to a social group (resp.~political party). 
Each $\PP$-formula is `tested' on {\em social situations} $(\beta, z) \in \mathbf{A}\times Z^S$ and each  $\SI$-formula on on {\em political situations} $(\beta, z) \in \mathbf{A}\times Z^P$, and the outcome of these `tests' is encoded into interpretation maps for each $\sigma\in \{\sigma_F,\sigma_D, \sigma_B\}$ and each $\pi\in \{\pi_L,\pi_C, \pi_U\}$ of the following types:
\begin{center}
	\begin{tabular}{ll}
		$\val{\pi}: \mathbf{A}\times Z^S\to \mathbf{A}$ & $\descr{\pi}: Z^S\to \mathbf{A}$\\
		$\val{\sigma}: \mathbf{A}\times Z^P\to \mathbf{A}$ & $\descr{\sigma}: Z^P\to \mathbf{A}$\\
	\end{tabular}
\end{center}
where as usual, for each formula $\phi$, the $\mathbf{A}$-set $\val{\phi}: \mathbf{A}\times Z\to \mathbf{A}$ indicates the extent to which  $\phi$ is supported on each situation of the appropriate type, and $\descr{\phi}: Z\to \mathbf{A}$ the extent to which it is rejected at each state of the appropriate type. In the setting of \cite{eusflat}, the intended interpretation of $\beta$ is the flexibility in translating/operationalizing data to variables. In analogy with this interpretation, we propose that when political promises are evaluated in situations $(\beta, z)$ where $z$ is a social groups, $\beta$ captures the maximum degree of flexibility in how voting (polling) is translated into the expression of the will of the group $z$. This degree might  include or take into account e.g.~the representativity of the sample, but also how rigorously the rules governing the test (e.g.\ eligibility criteria) are enforced, voter turnout, features of the electoral system like proportional representation vs first-past-the-post, etc. 
When social demands are evaluated in situation $(\beta, z)$ where $z$ is a political party, $\beta$ captures the maximum degree of flexibility in the outcome of the ``test'' is interpreted as the reflecting the position of the party  on that demand. Under a high $\beta$ value, one would be allowed to assign high significance to e.g.\ informal consultations among member of the party, while lower $\beta$ values would require higher standards of evidence, e.g.\ official policy documents of formal declarations following a party congress.

Let us represent $\mathbf{A}$-sets $\val{\phi}: \mathbf{A}\times Z\to \mathbf{A}$, for $Z \in \{Z^S, Z^P\}$, in tables with rows labelled by $\mathbf{A}$-elements and columns by $Z$-elements. Moreover, we represent $\descr{\phi}: Z \to \mathbf{A}$ as a triple $(\alpha, \beta, \gamma)$ where  $(\alpha, \beta, \gamma) = (\descr{\phi}(z_F), \descr{\phi}(z_D), \descr{\phi}(z_B))$ if $Z = Z^S$, and $(\alpha, \beta, \gamma) = (\descr{\phi}(z_L), \descr{\phi}(z_C), \descr{\phi}(z_X))$ if $Z = Z^P$. Then, using this notation, suppose that the interpretation of the political promises $\pi_L$ results in the following outcome: 

\begin{center}
	\begin{tabular}{l|lll}
		$\val{\pi_L}$ &$z_F$ &$z_D$ &$z_B$\\
		\hline
		$0.0$ &$0.6$ &$0.1$ &$0.2$\\
		$0.1$ &$0.7$ &$0.2$ &$0.3$\\
		$0.2$ &$0.8$ &$0.3$ &$0.4$\\
		$0.3$ &$0.9$ &$0.4$ &$0.5$\\
		$0.4$ &$1.0$ &$0.5$ &$0.6$\\
		$0.5$ &$1.0$ &$0.6$ &$0.7$\\
		$0.6$ &$1.0$ &$0.7$ &$0.8$\\
		$0.7$ &$1.0$ &$0.8$ &$0.9$\\
		$0.8$ &$1.0$ &$0.9$ &$1.0$\\
		$0.9$ &$1.0$ &$1.0$ &$1.0$\\
		$1.0$ &$1.0$ &$1.0$ &$1.0$\\
	\end{tabular}
	$\quad \quad$
	$\descr{\pi_L} = (0.4, 0.9, 0.8)$.
\end{center}
A noticeable feature of the table for $\val{\pi_L}$ is that, from the second row onwards, the value of any entry is always exactly $0.1$ greater than the entry one row above in the same column. This is no coincidence, and will be the case if the truth value algebra $\mathbf{A}$ is any finite subalgebra of the standard \L ukasiewics algebra. One can verify this by noting that, for any $\phi\in \mathcal{L}_{\mathsc{MT}}$, since $\val{\phi}$ is Galois-closed, it can be recovered from $\descr{\phi}$ as follows:
\begin{align*}
\val{\phi}(\beta, z) &= \bigwedge_{z'\in Z_X}[ \descr{\phi}(z')\to (E(z, z')\to\beta)]\\
&= \bigwedge_{z'\in Z_X}[ (\descr{\phi}(z') \otimes (E(z, z')) \to\beta)]\\
&= \bigvee_{z'\in Z_X} [\descr{\phi}(z') \otimes (E(z, z'))] \to\beta)\\
&= \min \left\{ 1, 1 - \left( \bigvee_{z'\in Z_X} [\descr{\phi}(z') \otimes (E(z, z'))]\right)  +  \beta \right\}\\
\end{align*}
As a result, we will represent the values of the other political promises and social demands more compactly, by giving only the first row of the table in each case. So, suppose that the interpretation of the political promises $\pi_C$ and $\pi_X$ results in the following outcomes: 

\begin{center}
	\begin{tabular}{l|lll}
		$\val{\pi_C}$ &$z_F$ &$z_D$ &$z_B$\\
		\hline
		$0.0$ &$0.2$ &$0.7$ &$0.7$\\
		$\vdots$ &$\vdots$ &$\vdots$ &$\vdots$
	\end{tabular}
	$\quad \quad$
	\begin{tabular}{l|lll}
		$\val{\pi_X}$ &$z_F$ &$z_D$ &$z_B$\\
		\hline
		$0.0$ &$0.6$ &$0.2$ &$0.4$\\
		$\vdots$ &$\vdots$ &$\vdots$ &$\vdots$
	\end{tabular}
\end{center}
with $\descr{\pi_C} = (0.8, 0.3, 0.3)$  and  $\descr{\pi_X} = (0.4, 0.8, 0.6)$. Suppose further that the interpretation of the social demand above results in the following outcomes: 

\begin{center}
	\begin{tabular}{l|lll}
		$\val{\sigma_F}$ &$z_L$ &$z_C$ &$z_X$\\
		\hline
		$0.0$ &$0.6$ &$0.3$ &$0.6$\\
		$\vdots$ &$\vdots$ &$\vdots$ &$\vdots$
	\end{tabular}
	$\quad \quad$
	\begin{tabular}{l|lll}
		$\val{\sigma_D}$ &$z_L$ &$z_C$ &$z_X$\\
		\hline
		$0.0$ &$0.3$ &$0.6$ &$0.3$\\
		$\vdots$ &$\vdots$ &$\vdots$ &$\vdots$
	\end{tabular}
	$\quad \quad$
	\begin{tabular}{l|lll}
		$\val{\sigma_B}$ &$z_L$ &$z_C$ &$z_X$\\
		\hline
		$0.0$ &$0.3$ &$0.6$ &$0.6$\\
		$\vdots$ &$\vdots$ &$\vdots$ &$\vdots$
	\end{tabular}
\end{center}
with  $\descr{\sigma_F} = (0.4, 0.7, 0.4)$, $\descr{\sigma_D} = (0.7, 0.4, 0.7)$  and  $\descr{\sigma_B} = (0.7, 0.4, 0.4)$. 

We are now in a position to compute the extensions of the $\SI$-formulas $\Diamond \pi_L$, $\Diamond \pi_C$ $\Diamond \pi_U$, and of the $\PP$-formulas $\lozenge \sigma_F$, $\lozenge \sigma_D$, $\lozenge \sigma_B$. We will only consider and interpret two examples. It can be verified that: 

\begin{center}
	\begin{tabular}{l|lll}
		$\val{\lozenge \sigma_D}$ &$z_F$ &$z_D$ &$z_B$\\
		\hline
		$0.0$ &$0.0$ &$0.3$ &$0.0$\\
		$\vdots$ &$\vdots$ &$\vdots$ &$\vdots$
	\end{tabular}
	$\quad \quad$
	\begin{tabular}{l|lll}
		$\val{\Diamond \pi_C}$ &$z_L$ &$z_C$ &$z_X$\\
		\hline
		$0.0$ &$0.0$ &$0.2$ &$0.1$\\
		$\vdots$ &$\vdots$ &$\vdots$ &$\vdots$
	\end{tabular}
\end{center}

Recall that $\val{\pi_C}$ measured the reaction from the three social groups to the promise by the Conservative Party to reduce the use of tax money to enforce the fox hunting ban. This was supported by both the Duke's social group and the bankers to the extent $0.7$ (with $\beta = 0$), presumably since both groups seek lower taxes and, in the first case, also presumably since fox hunting is part of British aristocratic culture. The factory workers' support was low, at $0.2$. Now $\val{\Diamond \pi_C}$ represents the response of the three political parties to this information:
notice that it is proportional to the extent to which the target demographics respond to the promise and the extent to which they are targeted by the parties. For example, given their main target demographics, namely the Duke's social group and the bankers, the Tories are more favourably inclined towards their own promise than is Labour, whose main target demographic, the factory workers, don't respond very well to this promise. The Brexit party lies in between, since there is a mixed response from their main target demographics, namely the factory workers and bankers.   

Turning to 	$\val{\lozenge \sigma_D}$, recall that $\val{\sigma_D}$ represents the degree of support the three political parties give (e.g.~as measured by statements of members of the party, policy documents etc.) to the social group of the Duke's demand for veto powers for the house of lords. The value of $\val{\lozenge \sigma_D}$ in turn represents the response of the three social groups to this support expressed by the political parties. Neither the factory workers nor the bankers show any enthusiasm, while the Duke's social group shows a more positive response, since their party of choice gives a fair measure of support to their demand.

\section{Epilogue}
We suggest that several interesting analogies can be drawn between competition of theories in the empirical sciences (cf.~\cite{eusflat}) and competition of social groups and political parties embodying socio-political theories;  these analogies can be drawn thanks to the general formal framework  adopted both in \cite{eusflat} and in the present paper, which we have illustrated with the case study discussed in the previous section.

In \cite{eusflat}, the competition of scientific theories (identified with sets of relevant variables) plays out in the arena of (a given graph of) databases, each of which is built according to a different theory, and therefore has different degrees of similarity to other databases in the graph.   Theories $X$ and $Y$ compete by having   their respective (key) hypotheses $\phi(X)$ and $\psi(Y)$ tested on all the databases of the given graph; each of these databases will be more or less suitable to test a given hypothesis. Hence, a clear-cut case in which  $X$ outcompetes  $Y$ is if, while each hypothesis is expected to score well on its `home-ground' (i.e.~on the databases built in accordance with the theory in the variables of which the given hypothesis is formulated, or maximally similar to those),  the performances of $\phi(X)$ on the databases that are not its own `home-ground' are better than the performances of  $\psi(Y)$ on the databases that are not its own `home-ground'.

Likewise, political competition between parties plays out in the arena of  (a given graph of) social groups, each of which has its own `social theory' (represented as the set of issues relevant to that social group), and therefore has different degrees of similarity to other social groups in the graph.   Parties $X$ and $Y$ compete by having   their respective (key) promises $\pi_1(X)$ and $\pi_2(Y)$ tested on all the social groups of the given graph; each of these social groups will be more or less receptive or supportive of a given promise. Hence, a clear-cut case in which  $X$ outcompetes  $Y$ is if, while each promise is expected to score well on its `home-ground'  (i.e.~on the social groups with strong affinity to the party),  the performances of $\pi_1(X)$ on the social groups that are not its own `home-ground' are better than the performances of  $\pi_2(Y)$ on the social groups that are not its own `home-ground', and this is decided by their respective performances on the social groups that are away-ground for both parties.

The socio-political competition has a further interesting twist, given by the possibility of simultaneously representing the competition between social groups  playing out in the arena of political parties. Again, a winner of this competition is a social group the demands of which are `listened to'  by a wider audience of political parties than their `home-ground'.  

Key to the possibility of winning on `away-ground' is a mechanism that is well known in the practice of science, and consists in the possibility of retrieving the values of variables that are not as such represented in the database by using ``translations'' of the values of other variables as {\em proxies}. We have proposed that certain unexpected socio-political alignments can be better understood in terms of an analogous mechanism in which issues that figure in the program of a political party can be translated into issues that figure on the agenda of social group, and vice versa.

Finally, although stylised and simplified, this framework  offers the possibility to analyse two competitive processes playing out at the same time, thereby paving the way to the possibility of formulating and answering a whole different range of formal questions about socio-political dynamics.

\section{Conclusions}

In this paper, we have introduced a many-valued  semantic environment for  a {\em multi-type} modal language based on the logic of general  (i.e.~not necessarily distributive) lattices. We have proved soundness and completeness for the basic logic, and, by means of a case study, we have illustrated the potential of this  framework as a tool for the formal analysis of socio-political competition. Below, we list some remarks about the present framework, and some further questions arising from this preliminary exploration.  


\paragraph{Expanding the language with fixed points.} Building on \cite{CCPZ}, `dual common knowledge' formulas such as $\mu X. \lozenge\Diamond (X\wedge \pi)$ and $\mu X. \Diamond\lozenge (X\wedge \sigma)$ can be understood as describing the convergence of ongoing processes of interaction between social groups and political parties. It would be interesting to  use the expressive power  of (multi-type) lattice-based fixed-point logic to describe and reason about these phenomena.

\paragraph{Towards an analysis of the dynamics of socio-political competition.}  Related to the previous point, the framework introduced in this paper lends itself to the formal analysis of the dynamics triggered by the interplay of social groups and political parties, a  theme on which recent research in political science has focused (cf.~e.g.~\cite{lowery2013policy,adapt-or-perish}). This direction would address questions relative e.g.~to the emergence of political parties in response to issues which are relevant to certain social groups, or to the emergence of novel social group identities by effect of certain political alignments. The framework also offers new connections to the dynamics of market categories \cite{wijnberg2011classification} and the appreciation of new members of old and new categories \cite{kuijken2017categorization} by different audiences.

\bibliography{BIBeusflat2019-2}

\begin{thebibliography}{10}

\bibitem{belohlavek}
R.~B{\^e}lohl{\'a}vek.
\newblock Fuzzy galois connections.
\newblock {\em Mathematical Logic Quarterly}, 45(4):497--504, 1999.

\bibitem{birkhoff1970heterogeneous}
Garrett Birkhoff and John~D Lipson.
\newblock Heterogeneous algebras.
\newblock {\em Journal of Combinatorial Theory}, 8(1):115--133, 1970.

\bibitem{conradie2015relational}
W.~Conradie and A.~Craig.
\newblock Relational semantics via {TiRS} graphs.
\newblock {\em Proc.~TACL 2015}, page long abstract, 2015.

\bibitem{eusflat}
W.~Conradie, A.~Craig, A.~Palmigiano, and N.~Wijnberg.
\newblock Modelling competing theories.
\newblock In {\em Proc.~EUSFLAT 2019}, Atlantis Studies in Uncertainty
  Modelling, page forthcoming, 2019.

\bibitem{graph-based-wollic}
W.~Conradie, A.~Craig, A.~Palmigiano, and N.M Wijnberg.
\newblock Modelling informational entropy.
\newblock In {\em Proc.~WoLLIC 2019}, volume 11541 of {\em Lecture Notes in
  Computer Science}, pages 140--160. Springer, 2019.

\bibitem{roughconcepts}
W.~Conradie, S.~Frittella, K.~Manoorkar, S.~Nazari, A.~Palmigiano,
  A.~Tzimoulis, and N.M. Wijnberg.
\newblock Rough concepts.
\newblock {\em Submitted}, ArXiv preprint: arXiv:1907.00359, 2019.

\bibitem{conradie2016categories}
W.~Conradie, S.~Frittella, A.~Palmigiano, M.~Piazzai, A.~Tzimoulis, and N.M.
  Wijnberg.
\newblock {Categories: How I Learned to Stop Worrying and Love Two Sorts}.
\newblock In {\em Proc.~WoLLIC 2016}, volume 9803 of {\em LNCS}, pages
  145--164, 2016.

\bibitem{Tarkpaper}
W.~Conradie, S.~Frittella, A.~Palmigiano, M.~Piazzai, A.~Tzimoulis, and N.M.
  Wijnberg.
\newblock Toward an epistemic-logical theory of categorization.
\newblock In {\em Proc.~TARK 2017}, volume 251 of {\em EPTCS}, pages 167--186,
  2017.

\bibitem{conradie2016constructive}
W.~Conradie and A.~Palmigiano.
\newblock Constructive canonicity of inductive inequalities.
\newblock {\em arXiv preprint arXiv:1603.08341}, 2016.

\bibitem{CoPa:non-dist}
W.~Conradie and A.~Palmigiano.
\newblock {Algorithmic correspondence and canonicity for non-distributive
  logics}.
\newblock {\em Annals of Pure and Applied Logic}, 170(9):923--974. DOI:
  10.1016/j.apal.2019.04.003, 2019.

\bibitem{conradie2018goldblatt}
W.~Conradie, A.~Palmigiano, and A.~Tzimoulis.
\newblock Goldblatt-{T}homason for {LE}-logics.
\newblock {\em arXiv preprint arXiv:1809.08225}, 2018.

\bibitem{CCPZ}
Willem Conradie, Andrew Craig, Alessandra Palmigiano, and Zhiguang Zhao.
\newblock Constructive canonicity for lattice-based fixed point logics.
\newblock In {\em Proc.~WoLLIC 2017}, volume 10388 of {\em Lecture Notes in
  Computer Science}, pages 92--109. Springer, 2017.
\newblock ArXiv preprint arXiv:1603.06547.

\bibitem{craig-gouveia-haviar}
A.~Craig, M.~Gouveia, and M.~Haviar.
\newblock {TiRS} graphs and {TiRS} frames: a new setting for duals of canonical
  extensions.
\newblock {\em Algebra universalis}, 74(1-2):123--138, 2015.

\bibitem{craig-priestley}
A.~Craig, M.~Haviar, and H.~Priestley.
\newblock A fresh perspective on canonical extensions for bounded lattices.
\newblock {\em Applied Categorical Structures}, 21(6):725--749, 2013.

\bibitem{hannan}
M.T.~Hannan et~al.
\newblock {\em Concepts and Categories: Foundations for Sociological and
  Cultural Analysis}.
\newblock Columbia University Press, 2019.

\bibitem{ganter2012formal}
B.~Ganter and R.~Wille.
\newblock {\em Formal concept analysis: mathematical foundations}.
\newblock Springer, 2012.

\bibitem{arjen}
C{\'e}sar Garc{\'\i}a-D{\'\i}az, Gilmar Zambrana-Cruz, and Arjen
  Van~Witteloostuijn.
\newblock Political spaces, dimensionality decline and party competition.
\newblock {\em Advances in Complex Systems}, 16(06):1350019, 2013.

\bibitem{greco2018algebraic}
G.~Greco, P.~Jipsen, F.~Liang, A.~Palmigiano, and A.~Tzimoulis.
\newblock Algebraic proof theory for {LE}-logics.
\newblock {\em arXiv preprint arXiv:1808.04642}, 2018.

\bibitem{ICLA2019paper}
G.~Greco, P.~Jipsen, K.~Manoorkar, A.~Palmigiano, and A.~Tzimoulis.
\newblock Logics for rough concept analysis.
\newblock In {\em Proc.~ICLA 2019}, volume 11600 of {\em LNCS}, pages 144--159,
  2019.

\bibitem{kuijken2017categorization}
B.~Kuijken, G.~Gemser, and N.~M. Wijnberg.
\newblock Categorization and willingness to pay for new products: The role of
  category cues as value anchors.
\newblock {\em Journal of Product Innovation Management}, 34(6):757--771, 2017.

\bibitem{lowery2013policy}
D.~Lowery, A.~van Witteloostuijn, G.~Peli, H.~Brasher, S.~Otjes, and
  S.~Gherghina.
\newblock Policy agendas and births and deaths of political parties.
\newblock {\em Party Politics}, 19(3):381--407, 2013.

\bibitem{nosofsky1991stimulus}
R.~M. Nosofsky.
\newblock Stimulus bias, asymmetric similarity, and classification.
\newblock {\em Cognitive Psychology}, 23(1):94--140, 1991.

\bibitem{pawlak}
Z.~Pawlak.
\newblock Rough set theory and its applications to data analysis.
\newblock {\em Cybernetics \& Systems}, 29(7):661--688, 1998.

\bibitem{tversky1977features}
A.~Tversky.
\newblock Features of similarity.
\newblock {\em Psychological review}, 84(4):327, 1977.

\bibitem{adapt-or-perish}
M.~van~de Wardt and A.~van Witteloostuijn.
\newblock Adapt or perish? how parties adapt to party system saturation in
  western democracies, 1945-2016.
\newblock {\em British Journal of Political Science}, page forthcoming, 2019.

\bibitem{wijnberg2011classification}
N.~M. Wijnberg.
\newblock Classification systems and selection systems: The risks of radical
  innovation and category spanning.
\newblock {\em Scandinavian Journal of Management}, 27(3):297--306, 2011.

\end{thebibliography}
\bibliographystyle{plain}

\appendix

\section{Completeness}
\label{sec:completeness}

For the sake of uniformity with previous settings (cf.~e.g.~\cite[Section 7.2]{roughconcepts}) in this section, we work with graph-based frames $\mathbb{G} = (\mathbb{X}_S, \mathbb{X}_P, R_{\lozenge}, R_{\Diamond})$ the associated complex algebras of which are different from those of Definition \ref{def:graph:based:frame:and:model}. That is, for the sake of this section, for every graph $\mathbb{X} = (E, Z)$, we define its associated formal context $\mathbb{P}_{\mathbb{X}}: = (Z_A, Z_X, I_{E})$ by setting $Z_A: = Z$, $Z_X: = \mathbf{A}\times Z$ and $I_{E}: Z_A\times Z_X\to \mathbf{A}$  be defined by the assignment $(z, (\alpha, z'))\mapsto E(z, z')\to \alpha$.

For any lattice $\mathbb{L}$, an $\mathbf{A}$-{\em filter} is an $\mathbf{A}$-subset of $\mathbb{L}$, i.e.~a map $f:\mathbb{L}\to \mathbf{A}$, which is both $\wedge$- and $\top$-preserving, i.e.~$f(\top) = 1$ and $f(a\wedge b) = f(a)\wedge f(b)$ for any $a, b\in\mathbb{L}$. Intuitively, the $\wedge$-preservation encodes a many-valued version of closure under $\wedge$ of filters. An $\mathbf{A}$-filter is {\em proper} if it is also $\bot$-preserving, i.e.~$f(\bot)= 0$.   Dually, an $\mathbf{A}$-{\em ideal} is a map $i:\mathbb{L}\to \mathbf{A}$ which is both $\vee$- and $\bot$-reversing, i.e.~$i(\bot) = \top$ and $i(a\vee b) = i(a)\wedge i(b)$ for any $a, b\in\mathbb{L}$, and is {\em proper} if in addition $i(\top) = 0$. The {\em complement} of a (proper) $\mathbf{A}$-ideal is a map $u: \mathbb{L}\to \mathbf{A}$  which is both $\vee$- and $\bot$-preserving, i.e.~$u(\bot) = 0$ and $u(a\vee b) = u(a)\vee u(b)$ for any $a, b\in\mathbb{L}$ (and in addition $u(\top) = 1$). Intuitively, $u(a)$ encodes the extent to which $a$ does not belong to the ideal of which $u$ is the many-valued complement.  We let $\mathsf{F}_{\mathbf{A}}(\mathbb{L})$, $\mathsf{I}_{\mathbf{A}}(\mathbb{L})$ and $\mathsf{C}_{\mathbf{A}}(\mathbb{L})$ respectively denote the set of  proper $\mathbf{A}$-filters,   proper $\mathbf{A}$-ideals, and  the complements of proper $\mathbf{A}$-ideals of $\mathbb{L}$. 
For any heterogeneous  $\mathcal{L}_{\mathrm{MT}}$-algebra $(\mathbb{L}_S, \mathbb{L}_P, \lozenge, \Diamond)$ (cf.~Definition \ref{def:hetalgebra}) and all 
$\mathbf{A}$-subsets $k:\mathbb{L}_P\to \mathbf{A}$ and $h:\mathbb{L}_S\to \mathbf{A}$, let $k^{-\Diamond}: \mathbb{L}_S\to \mathbf{A}$ and $h^{-\lozenge}: \mathbb{L}_P\to \mathbf{A}$ be defined as $k^{-\Diamond}(s)=\bigvee\{k(p)\mid \Diamond p\leq s\}$ and $h^{-\lozenge}(p)=\bigvee\{h(s)\mid \lozenge s\leq p\}$, respectively. Then,
by definition, $k(p)\leq k^{-\Diamond}(\Diamond p)$ and 
$h(s)\leq h^{-\lozenge}(\lozenge s)$ for every $p\in \mathbb{L}_P$ and $s\in \mathbb{L}_S$.
Let $(\mathbf{SD}, \mathbf{PP}, \lozenge, \Diamond)$ be the Lindenbaum--Tarski heterogeneous algebra associated with $\mathbf{L}$.
\begin{lemma}
\label{lemma:f minus diam is filter}
\begin{enumerate}
\item If $f:\mathbb{L}_P\to \mathbf{A}$ is  an $\mathbf{A}$-filter, then so is $f^{-\Diamond}$. 
\item If $g:\mathbb{L}_S\to \mathbf{A}$ is  an $\mathbf{A}$-filter, then so is $g^{-\lozenge}$.
\item If $f:\mathbf{PP}\to \mathbf{A}$ is  a proper $\mathbf{A}$-filter, then so is $f^{-\Diamond}$.
\item If $g:\mathbf{SD}\to \mathbf{A}$ is  a proper $\mathbf{A}$-filter, then so is $g^{-\lozenge}$.
\item If $\pi_1, \pi_2\in \mathbf{PP}$, then  $\pi_1\vee \pi_2 = \top$ implies that  $\pi_1 = \top$ or $\pi_2 = \top$.
\item If $\sigma_1, \sigma_2\in \mathbf{SD}$, then  $\sigma_1\vee \sigma_2 = \top$ implies that  $\sigma_1 = \top$ or $\sigma_2 = \top$.
\end{enumerate}
\end{lemma}  
\begin{proof} 
1. 
For all $s, t\in \mathbb{L}_S$, 
\begin{center}
\begin{tabular}{cll}
   $f^{-\Diamond}(\top)$ & = & $\bigvee\{f(p)\mid \Diamond p\leq \top\}$\\
   & = & $\bigvee\{f(p)\mid p\in \mathbb{L}\}$\\
   & = & $f(\top)$\\
   & = & $1$\\
\end{tabular}
\end{center}
\begin{center}
\begin{tabular}{cll}
   &$f^{-\Diamond}(s)\wedge f^{-\Diamond}(t)$\\
    = &$ \bigvee\{f(p_1)\mid \Diamond p_1\leq s\} \wedge \bigvee\{f(p_2)\mid \Diamond p_2\leq t\}$\\
        = &$ \bigvee\{f(p_1)\wedge f(p_2)\mid \Diamond p_1\leq s\mbox{ and }\Diamond p_2\leq b\}$ & frame-distributivity \\ 
             = &$ \bigvee\{f(p_1\wedge p_2)\mid \Diamond p_1\leq s\mbox{ and }\Diamond p_2\leq t\}$ & $f$ is an $\mathbf{A}$-filter\\
               = &$ \bigvee\{f(p)\mid \Diamond p\leq s\mbox{ and }\Diamond p\leq t\}$ & ($\ast$)\\
 = &$\bigvee\{f(p)\mid \Diamond p\leq s\wedge t\}$ & \\
=&$f^{-\Diamond}(s\wedge t)$,\\
 \end{tabular}
 \end{center}
 the equivalence marked with ($\ast$) being due to the fact that $\Diamond(p_1\wedge p_2)\leq \Diamond p_1\wedge \Diamond p_2$.

3. 
Let $f:\mathbf{PP}\to \mathbf{A}$ be  a proper $\mathbf{A}$-filter.
$f^{-\Diamond}(\bot) = \bigvee\{f([\pi])\mid [\Diamond \pi] \leq [\bot]\} = \bigvee\{f([\pi])\mid \Diamond \pi \vdash \bot \} = \bigvee\{f([\pi])\mid \pi \vdash \bot \} = f([\bot]) = 0$. The crucial inequality is the third to last, which holds since  $\Diamond \pi \vdash \bot$ iff $\pi \vdash \bot$. The right to left implication can be easily derived in 
$\mathbf{L}$.  
For the sake of the left to right implication we appeal to the completeness of 
$\mathbf{L}$ 
with respect to the class of all heterogeneous $\mathcal{L}_{\mathrm{MT}}$-algebras (cf.~Proposition \ref{prop:completeness}, see  \cite{CoPa:non-dist} for the general case) and reason contrapositively. Suppose $\pi \not\vdash \bot$. Then, by Proposition \ref{prop:completeness}, there is a normal heterogeneous  $\mathcal{L}_{\mathrm{MT}}$-algebra $\mathbb{H} = (\mathbb{L}_S, \mathbb{L}_P, \lozenge, \Diamond)$ and assignment $h$ of  atomic propositions   such that $h(\pi) \neq 0$. Now consider the heterogeneous algebra $(\mathbb{L}_S, \mathbb{L}'_P, \lozenge, \Diamond)$ obtained from $(\mathbb{L}_S, \mathbb{L}_P, \lozenge, \Diamond)$ by adding a new least element $0'$ to $\mathbb{L}_P$ and extending the $\Diamond$-operation by declaring $\Diamond 0' = 0'$. We keep the assignment $h$ unchanged. It is easy to check that $(\mathbb{L}_S, \mathbb{L}'_P, \lozenge, \Diamond)$ is a normal heterogeneous  $\mathcal{L}_{\mathrm{MT}}$-algebra, and that $h(\Diamond \pi) \geq 0 > 0'$ and hence $\Diamond \pi \not \vdash \bot$.

Items 2 and 4 are proven by  arguments analogous  to the ones above.

5. 
%
Suppose, by contraposition, that $\top \not\vdash \pi_1$ and $\top \not\vdash \pi_2$. 
By the completeness theorem to which we have appealed in the proof of item 2, there are heterogeneous algebras $\mathbb{H}_1 = (\mathbb{L}^S_1, \mathbb{L}^P_1, \lozenge_1, \Diamond_1)$ and $\mathbb{H}_2 = (\mathbb{L}^S_2, \mathbb{L}^P_2, \lozenge_2, \Diamond_2)$ and corresponding assignments $v_i$ on $\mathbb{H}_i$ such that $v_1(\pi_1) \neq \top^{\mathbb{L}^P_1}$ and $v_2(\pi_2) \neq \top^{\mathbb{L}^P_2}$. Consider the  algebra $\mathbb{H}' = (\mathbb{L}^S_1 \times \mathbb{L}^S_2, \mathbb{L}^{P'}, \Diamond', \lozenge')$, where $\mathbb{L}^{P'}$ is obtained  by adding a new top element $\top'$ to  $\mathbb{L}^P_1\times \mathbb{L}^P_2$, defining the operation $\Diamond' $ by the same assignment of $\Diamond^{\mathbb{H}_1\times \mathbb{H}_2}$ on $\mathbb{L}^P_1\times \mathbb{L}^P_2$ and mapping $\top'$ to $(\Diamond \top)^{\mathbb{L}^S_1\times \mathbb{L}^S_2}$, and the operation $\lozenge' $ by the same assignment of $\lozenge^{\mathbb{H}_1\times \mathbb{H}_2}$ on $\mathbb{L}^S_1\times \mathbb{L}^S_2$. The  monotonicity of $\Diamond'$ and normality (i.e.~finite join-preservation) of $\lozenge'$ follow immediately by construction. The normality (i.e.~finite join-preservation) of $\Diamond'$ is verified by cases: if $a\lor b\neq \top'$, then it immediately follows from the normality of $\Diamond^{\mathbb{H}_1\times\mathbb{H}_2}$.  If $a\lor b=\top'$, then by construction, either $a=\top'$ or $b=\top'$ (i.e.~$\top'$ is join-irreducible), and hence, the join-preservation of $\Diamond'$ is a consequence of its monotonicity. 
Consider  the valuation $v': \mathsf{Prop}\to \mathbb{H}'$ defined by the assignment $p\mapsto e (v_1(p), v_2(p))$, where $e: \mathbb{H}_1\times \mathbb{H}_2\to \mathbb{H}'$ is the natural embedding. 


Let us show, for all $\chi \in \PP$, that if $(v_1(\chi), v_2(\chi)) \neq \top^{\mathbb{L}^P_1 \times \mathbb{L}^P_2}$, then $v'(\chi) \neq \top'$. We proceed by induction on $\chi$. The cases for atomic propositions and conjunction are immediate. The case for $\chi: = \pi_1'\vee \pi_2'$ uses the join-irreducibility of $\top'$. When $\chi := \lozenge \sigma$, then $v'(\lozenge \sigma) = \lozenge' v'(\sigma)\neq \top'$, since,  by construction, $\top'$ is not in the range of $\lozenge'$.


Clearly, $v_1(\pi_1) \neq \top^{\mathbb{L}^P_1}$ and $v_2(\pi_2) \neq \top^{\mathbb{L}^P_2}$ imply that $(v_1(\pi_1), v_2(\pi_1)) \neq \top^{\mathbb{L}^P_1 \times \mathbb{L}^P_2}$ and $(v_1(\pi_2), v_2(\pi_2)) \neq \top^{\mathbb{L}^P_1 \times \mathbb{L}^P_2}$. So, by the above claim,  $v'(\pi_1) \neq \top'$ and $v'(\pi_2) \neq \top'$, and hence, since $\top'$ is join-irreducible, $v'(\phi \vee \psi) \neq \top'$.

The proof of item 6 is analogous to the one above.
 \end{proof}


\begin{lemma}\label{eq:premagicnew2} 
For any $f\in \mathsf{F}_{\mathbf{A}}(\mathbb{L}_P)$ and $v\in \mathsf{C}_{\mathbf{A}}(\mathbb{L}_S)$ and  $g\in \mathsf{F}_{\mathbf{A}}(\mathbb{L}_S)$  and $u\in \mathsf{C}_{\mathbf{A}}(\mathbb{L}_P)$,
\begin{enumerate}
\item $\bigwedge_{s\in \mathbb{L}_S}(f^{-\Diamond}(s)\to  v(s)) = \bigwedge_{p\in \mathbb{L}_P}(f(p)\to  v(\Diamond p))$;
 \item $\bigwedge_{p\in \mathbb{L}_P}(g^{-\lozenge}(p)\to  u(p)) = \bigwedge_{s\in \mathbb{L}_S}(g(s)\to  u(\lozenge s))$.
\end{enumerate}
\end{lemma}
\begin{proof} 1. The fact that
$f(p)\leq f^{-\Diamond}(\Diamond p)$ implies that  
$f^{-\Diamond}(\Diamond p)\to v(\Diamond p)\leq f(p)\to v(\Diamond p)$ for every $p\in \mathbb{L}_P$, which is enough to show that $\bigwedge_{s\in \mathbb{L}_S} (f^{-\Diamond}(s)\to v(s))\leq \bigwedge_{p\in \mathbb{L}_P} (f(p)\to v(\Diamond p))$. Conversely, to show that \[ \bigwedge_{p\in \mathbb{L}_P} (f(p)\to v(\Diamond p))\leq \bigwedge_{s\in \mathbb{L}_S} (f^{-\Diamond}(s)\to v(s)),\] 
we have to show that, for every $s\in \mathbb{L}_S$,
 \[\bigwedge_{p\in \mathbb{L}_P} (f(p)\to v(\Diamond p))\leq f^{-\Diamond}(s)\to v(s),\] 
 i.e.~by definition of $f^{-\Diamond}(s)$ and the fact that $\to$ is completely join-reversing in its first coordinate,
  \[\bigwedge_{p\in \mathbb{L}_P} (f(p)\to v(\Diamond p))\leq \bigwedge_{\Diamond q\leq s} (f(q)\to v(s)).\] 
  Hence, let $q\in \mathbb{L}_P$ such that $\Diamond q\leq s$, and let us show that   \[\bigwedge_{p\in \mathbb{L}_P} (f(p)\to v(\Diamond p))\leq f(q)\to v(s).\] 
 Since $v$ is $\vee$-preserving, hence order-preserving, $\Diamond q\leq s$ implies $v(\Diamond q)\leq v(s)$, hence
 \[\bigwedge_{p\in \mathbb{L}_P} (f(p)\to v(\Diamond p))\leq  f(q)\to v(\Diamond q) \leq  f(q)\to v(s),\]
 as required.  The proof of the second item is analogous and omitted.
\end{proof}
\begin{definition}
\label{def:canonical frame}
Let $(\mathbf{SD}, \mathbf{PP}, \lozenge, \Diamond)$ be the Lindenbaum-Tarski heterogeneous algebra  of $\mathcal{L}_{\mathrm{MT}}$-formulas.\footnote{In the remainder of this section, we abuse notation and identify formulas with their equivalence class in $(\mathbf{SD}, \mathbf{PP}, \lozenge, \Diamond)$. Also, notice the inversion: states in  $Z^S$ (resp.~$Z^P$) are built out of structures from  $\mathbf{PP}$ (resp.~$\mathbf{SD}$).}	The {\em canonical} graph-based $\mathbf{A}$-{\em frame} is the structure  $\mathbb{G} = (\mathbb{X}_S, \mathbb{X}_P, R_{\Diamond}, R_\lozenge)$  defined as follows:\footnote{Recall that for any set $W$, the $\mathbf{A}$-{\em subsethood} relation between elements of $\mathbf{A}$-subsets of $W$ is the map $S_W:\mathbf{A}^W\times \mathbf{A}^W\to \mathbf{A}$ defined as $S_W(f, g) :=\bigwedge_{w\in W}(f(w)\rightarrow g(w)) $. 
If $S_W(f, g) =1$ we also write $f\subseteq g$. 
}
	\[Z^S:=\Big\{(f,u)\in \mathsf{F}_{\mathbf{A}}(\mathbf{PP})\times  \mathsf{C}_{\mathbf{A}}(\mathbf{PP})\mid \bigwedge_{\pi\in \mathbf{PP}}(f(\pi)\to u(\pi))=1\Big\}.\]
	
	\[Z^P:=\Big\{(g, v)\in \mathsf{F}_{\mathbf{A}}(\mathbf{SD})\times  \mathsf{C}_{\mathbf{A}}(\mathbf{SD})\mid \bigwedge_{\sigma\in \mathbf{SD}}(g(\sigma)\to v(\sigma))=1\Big\}.\]
	For any $z\in Z^S$ (resp.~$z\in Z^P$) as above, we let $f_z$ and $u_z$ (resp.~$g_z$ and $v_z$) denote the first and second  coordinate of $z$, respectively. 
Then $E_P: Z^P\times Z^P\to \mathbf{A}$,  $E_S: Z^S\times Z^S\to \mathbf{A}$, $R_\lozenge: Z^S\times Z^P\to \mathbf{A}$ and $R_\Diamond: Z^P\times Z^S\to \mathbf{A}$ 
	are  defined as follows: \[E_S(z, z'): = \bigwedge_{\pi\in \mathbf{PP}}(f_z(\pi)\to u_{z'}(\pi));\]
	\[E_P(z, z'): = \bigwedge_{\sigma\in \mathbf{SD}}(g_z(\sigma)\to v_{z'}(\sigma));\]
	 \[R_\Diamond(z = (g_z, v_z), z' = (f_{z'}, u_{z'})) := \bigwedge_{\sigma\in \mathbf{SD}}(f_{z'}^{-\Diamond}(\sigma)\to v_z(\pi)) = \bigwedge_{\pi\in \mathbf{PP}}(f_{z'}(\pi)\to v_z(\Diamond\pi));\]
 \[R_\lozenge(z, z') := \bigwedge_{\pi\in \mathbf{PP}}(g_{z'}^{-\lozenge}(\pi)\to u_z(\pi)) = \bigwedge_{\sigma\in \mathbf{SD}}(g_{z'}(\sigma)\to u_z(\lozenge\sigma)).\]

	\end{definition}
	\begin{lemma}
	The structure $\mathbb{G}$ of Definition \ref{def:canonical frame}
 is a graph-based $\mathbf{A}$-frame, in the sense specified at the beginning of the present section.
	\end{lemma}
	\begin{proof}
We need to show that $R_{\Diamond}$ and $R_{\lozenge}$  satisfy the following compatibility conditions: for every $z\in Z^P$ and all $\alpha, \beta\in \mathbf{A}$,
\begin{align*}
(R_\Diamond^{[1]}[\{\beta / (\alpha, z) \}])^{[10]} &\subseteq R_\Diamond^{[1]}[\{\beta /  (\alpha, z) \}] \\
(R_\lozenge^{[0]}[\{\beta /  z \}])^{[01]} &\subseteq R_\lozenge^{[0]}[\{\beta / z \}],
\end{align*}
and for every $z\in Z^S$ and all $\alpha, \beta\in \mathbf{A}$,
\begin{align*}
(R_\lozenge^{[1]}[\{\beta / (\alpha, z) \}])^{[10]} &\subseteq R_\lozenge^{[1]}[\{\beta /  (\alpha, z) \}] \\
(R_\Diamond^{[0]}[\{\beta /  z \}])^{[01]} &\subseteq R_\Diamond^{[0]}[\{\beta / z \}].
\end{align*}
Let us show the fourth inclusion  above. 
By definition, for any $(\alpha, w)\in Z^P_X$,
\begin{center}
\begin{tabular}{r c l}
$R_{\Diamond}^{[0]}[\{\beta / z \}](\alpha, w)$ 
& = & $\bigwedge_{z'\in Z^S_A}[\{\beta / z \}(z')\to  (R_{\Diamond}(w, z')\to \alpha)]$\\
&  = & $\beta\to (R_{\Diamond}(w, z) \to \alpha)$\\
 & & \\
  $(R_\Diamond^{[0]}[\{\beta /  z \}])^{[01]} (\alpha, w)$ & = & $\bigwedge_{z'\in Z^P_A}[(R_\Diamond^{[0]}[\{\beta /  z \}])^{[0]}(z')\to (E_P(z', w)\to \alpha) ],$ \\
\end{tabular}
\end{center}
and hence it is enough to find some $z'\in Z^P_A$ such that 
\[(R_\Diamond^{[0]}[\{\beta /  z \}])^{[0]}(z')\to (E_P(z', w)\to \alpha)\leq \beta\to (R_{\Diamond}(w, z) \to \alpha),\]
i.e.
\begin{equation}
\label{eq:blurb}
\left(\bigwedge_{(\gamma, z'')\in Z^P_X}[\beta\to (R_\Diamond(z'', z)\to\gamma)]\to (E_P(z', z'')\to\gamma)\right) 
  \to (E_P(z', w)\to \alpha) \leq \beta\to (R_{\Diamond}(w, z) \to \alpha).  
\end{equation}
Let $z'\in Z_P^A$ such that $g_{z'} = f_z^{-\Diamond}$ (cf.~Lemma \ref{lemma:f minus diam is filter}). Then
\begin{center}
$E_P(z', w) = \bigwedge_{\sigma\in \mathbf{SD}}(f_z^{-\Diamond}(\sigma)\to v_{w}(\sigma)) =  R_{\Diamond}(w, z)$,
\end{center}
and likewise $E_P(z', z'') = R_{\Diamond}(z'', z)$. Therefore, for this choice of $z'$, inequality \eqref{eq:blurb} can be rewritten as follows:
\[\left(\bigwedge_{(\gamma, z'')\in Z^P_X}[\beta\to (R_\Diamond(z'', z)\to\gamma)]\to (R_{\Diamond}(z'', z)\to\gamma)\right)\to (R_{\Diamond}(w, z)\to \alpha) \leq \beta\to (R_{\Diamond}(w, z) \to \alpha)\]
The inequality above is true if \[\beta \leq \bigwedge_{(\gamma, z'')\in Z_X}[\beta\to (R_\Diamond(z'', z)\to\gamma)]\to (R_{\Diamond}(z'', z)\to\gamma),\]
i.e.~if for every $(\gamma, z'')\in Z^P_X$, 
\[\beta \leq [\beta\to (R_\Diamond(z'', z)\to\gamma)]\to (R_{\Diamond}(z'', z)\to\gamma),\]
which is an instance of a tautology in residuated lattices.

Let $z\in Z^P$ and  $\alpha, \beta\in \mathbf{A}$ and let us show that $(R_\Diamond^{[1]}[\{\beta / (\alpha, z) \}])^{[10]} \subseteq R_\Diamond^{[1]}[\{\beta /  (\alpha, z) \}]$. By definition, for every $w\in Z^S_A$,
\begin{center}
\begin{tabular}{r cl}
 $R_\Diamond^{[1]}[\{\beta /  (\alpha, z)\}] (w)$
&= &$\bigwedge_{(\gamma, z')\in Z^P_X}[\{\beta /  (\alpha, z)\}(\gamma, z')\to (R_{\Diamond}(z', w)\to \gamma)]$\\
&= &$\beta\to (R_{\Diamond}(z, w)\to \alpha)$\\
&&\\
 $(R_\Diamond^{[1]}[\{\beta / (\alpha, z) \}])^{[10]}(w)$
&= &$\bigwedge_{(\gamma, z')\in Z^S_X}[(R_\Diamond^{[1]}[\{\beta / (\alpha, z) \}])^{[1]}(\gamma, z')\to (E_S(w, z')\to \gamma)]$.\\
\end{tabular}
\end{center}
Hence it is enough to find some $(\gamma, z')\in Z^S_X$ such that 
\[(R_\Diamond^{[1]}[\{\beta / (\alpha, z) \}])^{[1]}(\gamma, z')\to (E_S(w, z')\to \gamma)\leq \beta\to ( R_{\Diamond}(z, w)\to \alpha),\]
i.e.
 \begin{equation}
 \label{eq:blurbb}
 \left( \bigwedge_{z''\in Z^S_A} (\beta\to (R_{\Diamond}(z, z'')\to \alpha))\to (E_S(z'', z')\to \gamma)\right)\to (E_S(w, z')\to \gamma)\leq \beta\to (R_{\Diamond}(z, w)\to \alpha).
\end{equation}

Let $\gamma: = \beta$, and $z'= (f_{z'}, u_{z'})\in Z^S$  such that $u_{z'}: \mathbf{PP}\to\mathbf{A}$ 
 is  defined by the assignment 
\[
u_{z'}(\pi) = \left\{\begin{array}{ll}
1 & \text{if }  \top\vdash \pi\\
v_z(\Diamond\pi) & \text{otherwise. } \\
\end{array}\right.
\]
By construction, $u_{z'}$  maps $\top$ to 1 and $\bot$ to 0; moreover, using Lemma \ref{lemma:f minus diam is filter}.5, it can be readily verified that $u_{z'}$ is $\vee$-preserving. Then, by Lemma \ref{eq:premagicnew2},

\[E_S(z'', z'): = \bigwedge_{\pi\in \mathbf{PP}}(f_{z''}(\pi)\to u_{z'}(\pi)) = \bigwedge_{\pi\in \mathbf{PP}}(f_{z''}(\pi)\to v_{z}(\Diamond\pi))  = \bigwedge_{\sigma\in \mathbf{SD}}(f_{z''}^{-\Diamond}(\sigma)\to v_z(\pi))  = R_{\Diamond}(z, z''),\]
and likewise $E(w, z') = R_{\Diamond}(z, w)$. Therefore, for this choice of $z'$, inequality \eqref{eq:blurbb} can be rewritten as follows:
\[\left( \bigwedge_{z''\in Z^S_A} (\beta\to (R_{\Diamond}(z, z'')\to \alpha))\to (R_{\Diamond}(z, z'')\to \gamma)\right) \to (R_{\Diamond}(z, w)\to \gamma)\leq \beta\to (R_{\Diamond}(z, w)\to \alpha),\]
which is shown to be true by the same argument as the one concluding the verification of  the previous inclusion. The remaining inclusions are verified with analogous arguments to those above (using Lemma \ref{lemma:f minus diam is filter}.6), and their proofs are omitted.
\end{proof}

\begin{definition}
\label{def:canonical model}
Let $(\mathbf{SD}, \mathbf{PP}, \lozenge, \Diamond)$ be the Lindenbaum-Tarski heterogeneous algebra  of $\mathcal{L}_{\mathrm{MT}}$-formulas.	The {\em canonical graph-based} $\mathbf{A}$-{\em model} is the structure  $\mathbb{M} =(\mathbb{G}, V)$ such that $\mathbb{G}$ is the canonical graph-based $\mathbf{A}$-frame of Definition \ref{def:canonical frame}, and  if $p \in \mathsf{Prop}$, then $V_S: \mathsf{Prop}\to \mathbb{X}_P^+$ and $V_P: \mathsf{Prop}\to \mathbb{X}_S^+$ are such that:
\begin{enumerate}
\item  $V_S(p) = (\val{p}_S, \descr{p}_S)$ with $\val{p}_S: Z^P_A\to \mathbf{A}$ and $\descr{p}_S: Z^P_X\to \mathbf{A}$ defined by  $z\mapsto g_z(p)$ and  $(\alpha, z)\mapsto v_z(p)\to \alpha$, respectively;
 
\item $V_P(p) = (\val{p}_P, \descr{p}_P)$ with $\val{p}_P: Z^S_A\to \mathbf{A}$ and $\descr{p}_P: Z^S_X\to \mathbf{A}$  defined by  $z\mapsto f_z(p)$ and  $(\alpha, z)\mapsto u_z(p)\to \alpha$, respectively.
\end{enumerate}
\end{definition}
\begin{lemma}
	The structure $\mathbb{G}$ of Definition \ref{def:canonical model}
 is a graph-based $\mathbf{A}$-model.
	\end{lemma}
	\begin{proof}
It is enough to show that for any $p \in \mathsf{Prop}$,
\begin{enumerate}
\item $\val{p}_P^{[1]}=\descr{p}_P$ and $\val{p}_P=\descr{p}_P^{[0]}$.
\item $\val{p}_S^{[1]}=\descr{p}_S$ and $\val{p}_S=\descr{p}_S^{[0]}$, and
\end{enumerate}  
We only show 1. To show that $\descr{p}_P (\alpha, z)\leq \val{p}_P^{[1]}(\alpha, z)$ for any $(\alpha, z)\in Z^S_X$, by definition, we need to show that
\[u_z(p)\to\alpha\leq \bigwedge_{z'\in Z^S_A}(\val{p}_P(z')\to (E_S(z', z)\to\alpha)),\]
i.e.~that for every $z'\in Z^S_A$, 
\[u_z(p)\to\alpha\leq \val{p}_P(z')\to (E_S(z', z)\to\alpha).\]
By definition, the inequality above is equivalent to
\[u_z(p)\to\alpha\leq f_{z'}(p)\to \left(\bigwedge_{\pi\in \mathbf{PP}} (f_{z'}(\pi)\to u_z(\pi))\to\alpha\right).\]
Since $\bigwedge_{\pi\in \mathbf{PP}} (f_{z'}(\pi)\to u_z(\pi))\leq f_{z'}(p)\to u_z(p)$ and $\to$ is order-reversing in its first coordinate, it is enough to show that 
\[u_z(p)\to\alpha\leq f_{z'}(p)\to [(f_{z'}(p)\to u_z(p))\to\alpha].\]
By residuation the inequality above is equivalent to
\[u_z(p)\to\alpha\leq [f_{z'}(p)\otimes (f_{z'}(p)\to u_z(p))]\to\alpha,\]
which is equivalent to
\[[f_{z'}(p)\otimes (f_{z'}(p)\to u_z(p)]\otimes [u_z(p)\to\alpha]\leq \alpha,\]
which is the instance of a tautology in residuated lattices.
Conversely, to show that $ \val{p}_P^{[1]}(\alpha, z)\leq \descr{p}_P (\alpha, z)$, i.e.
\[\bigwedge_{z'\in Z^S_A}(\val{p}_P(z')\to (E_S(z', z)\to\alpha))\leq u_z(p)\to\alpha,\]
it is enough to show that
\begin{equation}
\label{eqq}
\val{p}_P(z')\to (E_S(z', z)\to\alpha))\leq u_z(p)\to\alpha
\end{equation}
for some $z'\in Z^S$. Let $z': = (f_p, u)$ such that $u:\mathbf{PP}\to \mathbf{A}$ is the constant map $1$, and  $f_p:\mathbf{PP}\to \mathbf{A}$ is defined by the assignment 
\[
f_p(\pi) = \left\{\begin{array}{ll}
1 & \text{if }  p\vdash \pi\\
0 & \text{otherwise. } \\
\end{array}\right.
\]
 Hence, $E_S(z', z) = \bigwedge_{\pi\in \mathbf{PP}}(f_p(\pi)\to u_z(\pi))=\bigwedge_{p\vdash\pi}u_z(\pi) = u_z(p)$, the last identity holding since $u_z$ is order-preserving.  
 Therefore,  
 $\val{p}_P(z')\to (E_S(z', z)\to\alpha)) = f_p(p)\to (u_z(p)\to \alpha) = 1\to (u_z(p)\to \alpha) = u_z(p)\to \alpha$, which shows  \eqref{eqq}.
 
By adjunction, the inequality $\descr{p}_P \leq \val{p}_P^{[1]}$ proven above implies that   $\val{p}_P\leq\descr{p}_P^{[0]}$. Hence, to show that  $\val{p}_P=\descr{p}_P^{[0]}$, it is enough to show that  $\descr{p}_P^{[0]}(z)\leq\val{p}_P(z)$ for every $z\in Z^S$, i.e. 
\[\bigwedge_{(\alpha, z')\in Z^S_X} \descr{p}_P(\alpha, z')\to (E_S(z, z')\to \alpha)\leq f_z(p),\]
and to show the inequality above, it is enough to show that 
\begin{equation}\label{eqqq} \descr{p}_P(\alpha, z')\to (E_S(z, z')\to \alpha)\leq f_z(p)\end{equation}
for some $(\alpha, z')\in Z^S_X$.
Let $\alpha: =f_z(p)$ and $z':  = (f_{z'}, u_{p})$ be such that 
$u_{z'} = u_{p}: \mathbf{PP}\to \mathbf{A}$ is defined by the following assignment:
 \[
u_{p}(\pi) = \left\{\begin{array}{ll}
0 & \text{if } \pi\vdash \bot\\
f_z(p) & \text{if } \pi\vdash p \mbox{ and } \pi \not\vdash \bot\\
1 & \text{if } \pi\not\vdash p.
\end{array}\right.
\]
By construction, $u_{z'}$ is $\vee$-, $\bot$- and $\top$-preserving. Moreover, $\descr{p}_P(\alpha, z') = u_{z'}(p)\to\alpha = f_z(p)\to f_z(p) = 1$, and $E_S(z, z') = \bigwedge_{\pi\in \mathbf{PP}}(f_z(\pi)\to u_{z'}(\pi))  = \bigwedge_{\pi\vdash p}(f_z(\pi)\to f_z(p))  = 1$. 
 Hence, the left-hand side of \eqref{eqqq} can be equivalently rewritten as $1\to (1\to f_z(p)) = f_z(p)$, which shows \eqref{eqqq} and concludes the proof.
\end{proof}
\begin{lemma}[Truth Lemma]
\label{lemma:truthlemma}
	For every  $\pi\in \mathsf{PP}$ and every $\sigma\in \mathsf{SD}$,
	\begin{enumerate}
\item  the maps $\val{\pi}_P: Z^S_A\to \mathbf{A}$ and $\descr{\pi}_P: Z^S_X\to \mathbf{A}$ coincide with those defined by the assignments $z\mapsto f_z(\pi)$ and $(\alpha, z)\mapsto u_z(\pi)\to \alpha$, respectively. 
\item  the maps $\val{\sigma}_S: Z^P_A\to \mathbf{A}$ and $\descr{\sigma}_S: Z^P_X\to \mathbf{A}$ coincide with those defined by the assignments $z\mapsto g_z(\sigma)$ and $(\alpha, z)\mapsto v_z(\sigma)\to \alpha$, respectively. 
\end{enumerate}
\end{lemma}
\begin{proof}
We proceed by simultaneous induction on $\pi$ and $\sigma$. If $\pi: = p\in \mathsf{Prop}$ (resp.~$\sigma: = p\in \mathsf{Prop}$), the statement follows immediately from  Definition \ref{def:canonical model}. 

If $\pi: =\top$, then $\val{\top}_P(z) = 1 = f_z(\top)$ since $\mathbf{A}$-filters are $\top$-preserving. Moreover,
\begin{center}
\begin{tabular}{r cl}
 $\descr{\top}_P(\alpha, z)$& $=$& $\val{\top}_P^{[1]}(\alpha, z)$\\
 & $=$& $ \bigwedge_{z'\in Z^S_A} [\val{\top}(z')\to (E_S(z', z)\to\alpha)]$\\
  & $=$& $ \bigwedge_{z'\in Z^S_A} [f_{z'}(\top)\to (E_S(z', z)\to\alpha)]$\\
  & $=$& $ \bigwedge_{z'\in Z^S_A} [E_S(z', z)\to\alpha]$. \\
  \end{tabular}
 \end{center}
 So, to show that $u_z(\top)\to \alpha\leq \descr{\top}_P(\alpha, z)$, we need to show that for every $z'\in Z^S_A$,
\[u_z(\top)\to \alpha\leq E_S(z', z)\to\alpha,\]
and for this, it is enough to show that 
\[\bigwedge_{\pi'\in \mathbf{PP}}[f_{z'}(\pi')\to u_z(\pi')]\leq u_z(\top),\]
which is true, since by definition, $u_z(\top) = 1$.
To show that $\descr{\top}_P(\alpha, z)\leq  u_z(\top)\to \alpha$, i.e.~that \[\bigwedge_{z'\in Z^S_A} [E_S(z', z)\to\alpha]\leq u_z(\top)\to \alpha,\]
it is enough  to find some $z'\in Z^S$ such that  $E_S(z', z)\to\alpha\leq u_z(\top)\to \alpha$. Let $z': = (f_\top, u)$ such that $u:\mathbf{PP}\to \mathbf{A}$  maps $\top$ to $1$ and every other element of $\mathbf{PP}$ to $0$, and  $f_\top:\mathbf{PP}\to \mathbf{A}$ is defined by the assignment 
\[
f_\top(\pi') = \left\{\begin{array}{ll}
1 & \text{if }  \top\vdash \pi'\\
0 & \text{otherwise. } \\
\end{array}\right.
\]\
By definition, $E_S(z', z) = \bigwedge_{\pi'\in \mathbf{PP}}[f_{z'}(\pi')\to u_z(\pi')] =  \bigwedge_{\top\vdash \pi'}[1\to u_z(\pi')]=  \bigwedge_{\top\vdash \pi'}u_z(\pi')
\geq u_z(\top)$, the last inequality being due to the fact that  $u_z$ is order-preserving. 
Hence, $E_S(z', z)\to\alpha\leq u_z(\top)\to \alpha$, as required. The case in which $\sigma: = \top$ is analogous to the one above, and its proof is omitted.

If $\pi: =\bot$, then $\descr{\bot}_P(\alpha, z) = 1 = u_z(\bot)\to \alpha$ since complements of $\mathbf{A}$-ideals are $\bot$-preserving. Let us show that $\val{\bot}_P(z) = f_z(\bot)$. The inequality $f_z(\bot)\leq \val{\bot}_P(z) $ follows immediately from the fact that $f_z$ is a proper $\mathbf{A}$-filter and hence $f_z(\bot) = 0$. To show that $\val{\bot}_P(z) \leq f_z(\bot)$, by definition  $\val{\bot}_P(z)=\descr{\bot}^{[0]}(z) =  \bigwedge_{(\alpha, z')\in Z^S_X} [(u_{z'}(\bot)\to\alpha)\to (E_S(z, z')\to\alpha)]$,
%
hence, it is enough to find some $(\alpha, z')\in Z^S_X$ such that  \begin{equation}\label{eqqqqqq}(u_{z'}(\bot)\to\alpha)\to (E_S(z, z')\to\alpha) \leq f_z(\bot).\end{equation}
Let $\alpha: = f_z(\bot)$ and  let $z': = (f_{\top}, u_{\bot})$ such that $f_{\top}:\mathbf{PP}\to \mathbf{A}$ is defined as indicated above in the base case for $\pi: = \top$, and  $u_{\bot}:\mathbf{PP}\to \mathbf{A}$ is defined by the assignment 
	\[
u_{\bot}(\pi') = \left\{\begin{array}{ll}
0 & \text{if } \pi'\vdash \bot\\
1 & \text{if }\pi'\not\vdash \bot.
\end{array}\right.
\]
By definition and since $f_z$ is order-preserving and $\bot$-preserving, $E_S(z, z') = \bigwedge_{\pi'\in \mathbf{PP}}[f_z(\pi')\to u_{\bot}(\pi')]  = 1$. Hence, \eqref{eqqqqqq} can be rewritten as follows:
\[(f_z(\bot)\to f_z(\bot))\to f_z(\bot)\leq f_z(\bot),\]
which is true since $f_z(\bot)\to f_z(\bot) =1$ and $1\to f_z(\bot) = f_z(\bot)$.
The case in which $\sigma: = \bot$ is analogous to the one above, and its proof is omitted.

If  $\pi: = \pi_1\wedge\pi_2$, then $\val{\pi_1\wedge\pi_2}_P(z) = (\val{\pi_1}_P \wedge \val{\pi_2})_P(z) = \val{\pi_1}_P(z) \wedge \val{\pi_2}_P(z) = f_z(\pi_1) \wedge f_z(\pi_2) = f_z(\pi_1 \wedge \pi_2)$. 
Let us  show that $\descr{\pi_1 \wedge \pi_2}_P(\alpha, z) = u_z(\pi_1 \wedge \pi_2)\to \alpha$. By definition,
\begin{center}
\begin{tabular}{cl}
& $\descr{\pi_1 \wedge \pi_2}_P(\alpha, z)$\\
= & $\val{\pi_1 \wedge \pi_2}_P^{[1]}(\alpha, z)$\\
= & $\bigwedge_{z'\in Z^S_A}[\val{\pi_1 \wedge \pi_2}_P(z')\to (E_S(z', z)\to \alpha)]$\\
= & $\bigwedge_{z'\in Z^S_A}[f_{z'}(\pi_1 \wedge \pi_2)\to (E_S(z', z)\to \alpha)]$.\\
\end{tabular}
\end{center}
Hence,  to show that $u_z(\pi_1 \wedge \pi_2)\to \alpha\leq \descr{\pi_1 \wedge \pi_2}_P(\alpha, z)$, we need to show that for every $z'\in Z^S_A$,
\[u_z(\pi_1 \wedge \pi_2)\to \alpha\leq f_{z'}(\pi_1\wedge \pi_2)\to (E_S(z', z)\to \alpha).\]
Since by definition $E_S(z', z) = \bigwedge_{\pi'\in \mathbf{PP}}[f_{z'}(\pi')\to u_z(\pi')]\leq f_{z'}(\pi_1 \wedge \pi_2)\to u_z(\pi_1 \wedge \pi_2)$ and $\to$ is order-reversing in the first coordinate and order-preserving in the second one, 
it is enough to show that for every $z'\in Z^S_A$,
\[u_z(\pi_1 \wedge \pi_2)\to \alpha \leq f_{z'}(\pi_1\wedge \pi_2)\to ((f_{z'}(\pi_1 \wedge \pi_2)\to u_z(\pi_1 \wedge \pi_2))\to \alpha).\]

By residuation, the above inequality is equivalent to
\[u_z(\pi_1 \wedge \pi_2)\to \alpha \leq[f_{z'}(\pi_1\wedge \pi_2)\otimes (f_{z'}(\pi_1 \wedge \pi_2)\to u_z(\pi_1 \wedge \pi_2))]\to \alpha.\]
The above inequality is true if 
\[f_{z'}(\pi_1\wedge \pi_2)\otimes (f_{z'}(\pi_1 \wedge \pi_2)\to u_z(\pi_1 \wedge \pi_2))\leq u_z(\pi_1 \wedge \pi_2),\]
which is an instance of a tautology in residuated lattices.

To show that $\descr{\pi_1 \wedge \pi_2}_P(\alpha, z)\leq u_z(\pi_1 \wedge \pi_2)\to \alpha$, it is enough to find some $z'\in Z^S_A$ such that 
\[f_{z'}(\pi_1 \wedge \pi_2)\to (E_S(z', z)\to \alpha)\leq u_z(\pi_1 \wedge \pi_2)\to \alpha.\]
Let $z': = (f_{\pi_1 \wedge \pi_2}, u_\bot)$ such that $u_\bot:\mathbf{PP}\to \mathbf{A}$ 
is defined as indicated above in the base case for $\pi: = \bot$,
and  $f_{\pi_1 \wedge \pi_2}:\mathbf{PP}\to \mathbf{A}$ is defined by the assignment 
\[
f_{\pi_1 \wedge \pi_2}(\pi') = \left\{\begin{array}{ll}
1 & \text{if }  \pi_1 \wedge \pi_2\vdash \pi'\\
0 & \text{otherwise. } \\
\end{array}\right.
\]

For $z': = z$, since $f_{z'}(\pi_1 \wedge \pi_2) = 1$ and $1\to (E_S(z', z)\to \alpha) = E_S(z', z)\to \alpha$,  the inequality above becomes
\[E_S(z', z)\to \alpha\leq u_z(\pi_1 \wedge \pi_2)\to \alpha,\]
to verify which, it is enough to show that $u_z(\pi_1 \wedge \pi_2)\leq E_S(z', z)$. Indeed, 
by definition, $E_S(z', z) =  \bigwedge_{\pi'\in \mathbf{PP}}[f_{z'}(\pi')\to u_z(\pi')] =  \bigwedge_{\pi_1 \wedge \pi_2\vdash \pi'}[1\to u_z(\pi')]=  \bigwedge_{\pi_1 \wedge \pi_2\vdash \pi'}u_z(\pi')\geq u_z(\pi_1 \wedge \pi_2)$, the last inequality being due to the fact that  $u_z$ is order-preserving. The case in which $\sigma: = \sigma_1\wedge \sigma_2$ is analogous to the one above, and its proof is omitted.

\medskip

If  $\pi: = \pi_1\vee\pi_2$, then $\descr{\pi_1\vee\pi_2}_P(\alpha, z) = (\descr{\pi_1}_P \wedge \descr{\pi_2}_P)(\alpha, z) = \descr{\pi_1}_P(\alpha, z) \wedge \descr{\pi_2}_P(\alpha, z) = (u_z(\pi_1)\to\alpha) \wedge (u_z(\pi_2)\to \alpha) = (u_z(\pi_1)\vee u_z(\pi_2))\to \alpha)=  u_z(\pi_1 \vee \pi_2)\to \alpha$. 
Let us  show that $\val{\pi_1 \vee \pi_2}_P(z) = f_z(\pi_1 \vee \pi_2)$. By definition,
\begin{center}
\begin{tabular}{rcl}
 $\val{\pi_1 \vee \pi_2}_P(z)$
&= & $\descr{\pi_1 \vee \pi_2}_P^{[0]}(z)$\\
&= & $\bigwedge_{(\alpha, z')\in Z^S_X}[\descr{\pi_1 \vee \pi_2}_P(\alpha, z')\to (E_S(z, z')\to \alpha)]$\\
&= & $\bigwedge_{(\alpha, z')\in Z^S_X}[(u_{z'}(\pi_1 \vee \pi_2)\to\alpha)\to (E_S(z, z')\to \alpha)]$.\\
\end{tabular}
\end{center}
Hence,  to show that $f_z(\pi_1 \vee \pi_2)\leq \val{\pi_1 \vee \pi_2}(z)$, we need to show that for every $(\alpha, z')\in Z^S_X$,
\[f_z(\pi_1 \vee \pi_2)\leq (u_{z'}(\pi_1 \vee \pi_2)\to\alpha)\to (E_S(z, z')\to \alpha).\]
Since by definition $E_S(z, z') = \bigwedge_{\pi'\in \mathbf{PP}}[f_{z}(\pi')\to u_{z'}(\pi')]\leq f_{z}(\pi_1 \vee \pi_2)\to u_{z'}(\pi_1 \vee \pi_2)$ and $\to$ is order-reversing in the first coordinate and order-preserving in the second one, 
it is enough to show that for every $(\alpha, z')\in Z^S_X$,
\[f_z(\pi_1 \vee \pi_2)
\leq (u_{z'}(\pi_1 \vee \pi_2)\to\alpha)\to ((f_{z}(\pi_1 \vee \pi_2)\to u_{z'}(\pi_1 \vee \pi_2))\to \alpha).\]
By residuation, associativity and commutativity of $\otimes$, the inequality above is equivalent to
\[f_z(\pi_1 \vee \pi_2)\otimes (f_{z}(\pi_1 \vee \pi_2)\to u_{z'}(\pi_1 \vee \pi_2)) \otimes (u_{z'}(\pi_1 \vee \pi_2)\to\alpha) \leq \alpha,\]
which is a tautology in residuated lattices.

To show that $\val{\pi_1 \vee \pi_2}_P(z)\leq f_z(\pi_1 \vee \pi_2)$, it is enough to find some $(\alpha, z')\in Z^S_X$ such that 
\begin{equation}
\label{eqqqqq}
(u_{z'}(\pi_1 \vee \pi_2)\to\alpha)\to (E_S(z', z)\to \alpha)\leq f_z(\pi_1 \vee \pi_2).\end{equation}
Let $\alpha: = f_z(\pi_1 \vee \pi_2)$ and  let $z': = (f_{\top}, u_{\pi_1 \vee \pi_2})$ such that $f_{\top}:\mathbf{PP}\to \mathbf{A}$ is defined as indicated above in the base case for $\pi:= \top$,
and  $u_{\pi_1 \vee \pi_2}:\mathbf{PP}\to \mathbf{A}$ is defined by the assignment 
	\[
u_{\pi_1 \vee \pi_2}(\pi') = \left\{\begin{array}{ll}
0 & \text{if } \pi'\vdash \bot\\
f_z(\pi_1 \vee \pi_2) & \text{if }  \pi'\not\vdash \bot \mbox{ and } \pi'\vdash \pi_1 \vee \pi_2 \\
1 & \text{if }\pi'\not\vdash \pi_1 \vee \pi_2.
\end{array}\right.
\]
By definition and since $f_z$ is order-preserving and proper, $E_S(z, z') = \bigwedge_{\pi'\in \mathbf{PP}}[f_z(\pi')\to u_{\pi_1 \vee \pi_2}(\psi)] = \bigwedge_{\bot\not\dashv \pi'\vdash\pi_1 \vee \pi_2} [f_z(\pi')\to f_z(\pi_1 \vee \pi_2) ] = 1$. Hence, \eqref{eqqqqq} can be rewritten as follows:
\[(f_z(\pi_1 \vee \pi_2)\to f_z(\pi_1 \vee \pi_2))\to f_z(\pi_1 \vee \pi_2)\leq f_z(\pi_1 \vee \pi_2),\]
which is true since $f_z(\pi_1 \vee \pi_2)\to f_z(\pi_1 \vee \pi_2) =1$ and $1\to f_z(\pi_1 \vee \pi_2) = f_z(\pi_1 \vee \pi_2)$. The case in which $\sigma: = \sigma_1\vee\sigma_2$ is analogous to the one above, and its proof is omitted. 

\medskip 
If $\sigma: = \Diamond\pi$, let us show that $\descr{\Diamond\pi}_S (\alpha, z) = v_z(\Diamond \psi)\to \alpha$ for any $(\alpha, z)\in Z^P_X$. By definition,
\begin{center}
\begin{tabular}{rcl}
$\descr{\Diamond\pi}_S (\alpha, z)$ &
 = & $R^{[0]}_\Diamond[\val{\pi}_P](\alpha, z)$\\
&= & $\bigwedge_{z'\in Z^S_A}[\val{\pi}_P(z')\to  (R_{\Diamond}(z, z') \to \alpha)]$\\
&= & $\bigwedge_{z'\in Z^S_A}[f_{z'}(\pi)\to  (R_{\Diamond}(z, z') \to \alpha)]$.\\
\end{tabular}
\end{center}
Hence,  to show that $v_z(\Diamond\pi)\to \alpha\leq \descr{\Diamond\pi}_S(\alpha, z)$, we need to show that for every $z'\in Z^S_A$,
\[v_z(\Diamond\pi)\to \alpha\leq f_{z'}(\pi)\to  (R_{\Diamond}(z, z') \to \alpha).\]
By definition and Lemma \ref{eq:premagicnew2},  $R_\Diamond(z, z') = \bigwedge_{\pi'\in \mathbf{PP}}(f_{z'}(\pi')\to v_z(\Diamond\pi')) \leq f_{z'}(\pi)\to v_{z}(\Diamond \pi)$, 
and since $\to$ is order-reversing in the first coordinate and order-preserving in the second one, 
it is enough to show that for every $z'\in Z^S_A$,
\[v_z(\Diamond\pi)\to \alpha\leq f_{z'}(\pi)\to  ((f_{z'}(\pi)\to v_{z}(\Diamond \pi)) \to \alpha).\]
By residuation, associativity and commutativity of $\otimes$, the inequality above is equivalent to
\[ [f_{z'}(\pi)\otimes (f_{z'}(\pi)\to v_{z}(\Diamond \pi))] \otimes (v_z(\Diamond\pi)\to \alpha) \leq \alpha,\]
which is a tautology in residuated lattices.

To show that $\descr{\Diamond\pi}_S(\alpha, z)\leq v_z(\Diamond\pi)\to \alpha$, it is enough to find some $z'\in Z^S_A$ such that
\begin{equation}
\label{eqqqqqqq}
f_{z'}(\pi)\to  (R_{\Diamond}(z, z') \to \alpha)\leq v_z(\Diamond\pi)\to \alpha.\end{equation}
Let $z': = (f_{\pi}, u_\bot)$ such that $u_\bot:\mathbf{PP}\to \mathbf{A}$ 
is defined as indicated above in the base case for $\pi: = \bot$,
 and  $f_{\pi}:\mathbf{PP}\to \mathbf{A}$ is defined by the assignment 
\[
f_{\pi}(\pi') = \left\{\begin{array}{ll}
1 & \text{if }  \pi\vdash \pi'\\
0 & \text{otherwise. } \\
\end{array}\right.
\]
By definition and Lemma \ref{eq:premagicnew2},
\begin{center}
\begin{tabular}{rcl}
$R_\Diamond(z, z')$ 
&=&$ \bigwedge_{\pi'\in \mathbf{PP}}(f_{z'}(\pi')\to  v_{z}(\Diamond \pi'))$\\
&=&$ \bigwedge_{\pi\vdash \pi'} v_{z}(\Diamond \pi')$\\
&$\geq$&$  v_{z}(\Diamond \pi)$,\\
\end{tabular}
\end{center}
the last inequality being due to the fact that $v_z$ and $\Diamond$ are order-preserving. Since $\to$ is order reversing in the first coordinate and order-preserving in the second one, to show \eqref{eqqqqqqq} it is enough to show that 
\[f_{z'}(\pi)\to  (v_{z}(\Diamond \pi) \to \alpha)\leq v_z(\Diamond\pi)\to \alpha.\]
This immediately follows from the fact that, by construction,  $f_{z'}(\pi) = 1$.

Let us show that $\val{\Diamond\pi}_S(z) = g_z(\Diamond\pi)$ for every $z\in Z^P_A$. By definition,
\begin{center}
\begin{tabular}{r cl}
$\val{\Diamond\pi}_S (z)$
&  = & $\descr{\Diamond\pi}_S^{[0]}(z)$\\ 
& = & $\bigwedge_{(\alpha, z')\in Z^P_X}[\descr{\Diamond\pi}(\alpha, z')\to (E_P(z, z')\to\alpha)]$\\
& = & $\bigwedge_{(\alpha, z')\in Z^P_X}[(v_{z'}(\Diamond\pi)\to \alpha)\to (E_P(z, z')\to\alpha)].$\\
\end{tabular}
\end{center}
Hence,  to show that $g_z(\Diamond\pi)\leq \val{\Diamond\pi}_S(z)$, we need to show that for every $(\alpha, z')\in Z^P_X$,
\[g_z(\Diamond\pi)\leq (v_{z'}(\Diamond\pi)\to \alpha)\to (E_P(z, z')\to\alpha).\]
Since by definition $E_P(z, z') = \bigwedge_{\sigma'\in \mathbf{SD}}[g_{z}(\sigma')\to v_{z'}(\sigma')]\leq g_{z}(\Diamond\pi)\to v_{z'}(\Diamond\pi)$ and $\to$ is order-reversing in the first coordinate and order-preserving in the second one, 
it is enough to show that for every $(\alpha, z')\in Z^P_X$,
\[g_z(\Diamond\pi)\leq (v_{z'}(\Diamond\pi)\to \alpha)\to ((g_{z}(\Diamond\pi)\to v_{z'}(\Diamond\pi))\to\alpha).\]
By residuation, associativity and commutativity of $\otimes$, the inequality above is equivalent to
\[ [g_z(\Diamond\pi)\otimes (g_{z}(\Diamond\pi)\to v_{z'}(\Diamond\pi))] \otimes (v_{z'}(\Diamond\pi)\to \alpha) \leq \alpha,\]
which is a tautology in residuated lattices.

To show that $\val{\Diamond\pi}_S(z)\leq g_z(\Diamond\pi)$, it is enough to find  some $(\alpha, z')\in Z^P_X$ such that 
\begin{equation}
\label{eqqqqqqqq}(u_{z'}(\Diamond\pi)\to \alpha)\to (E_P(z, z')\to\alpha)\leq g_z(\Diamond\pi).\end{equation}
Let $\alpha: = g_z(\Diamond\pi)$ and 
  let $z': = (g_{\top}, v_{\Diamond\pi})$ such that $g_{\top}:\mathbf{SD}\to \mathbf{A}$ maps $\top$ to $1$ and every other element of $\mathbf{SD}$ to $0$, and  $v_{\Diamond\pi}:\mathbf{SD}\to \mathbf{A}$ is defined by the assignment 
	\[
v_{\Diamond\pi}(\sigma') = \left\{\begin{array}{ll}
0 & \text{if } \sigma'\vdash \bot\\
g_z(\Diamond\pi) & \text{if }  \sigma'\not\vdash \bot \mbox{ and } \sigma'\vdash \Diamond\pi \\
1 & \text{if }\sigma'\not\vdash \Diamond\pi.
\end{array}\right.
\]
By definition and since $g_z$ is order-preserving and proper, $E_P(z, z') = \bigwedge_{\sigma'\in \mathbf{SD}}[g_z(\sigma')\to v_{\Diamond\pi}(\sigma')] = \bigwedge_{\sigma'\vdash\Diamond\pi} [g_z(\sigma')\to g_z(\Diamond\psi) ] = 1$. Hence, \eqref{eqqqqqqqq} can be rewritten as follows:
\[(g_z(\Diamond\pi)\to g_z(\Diamond\pi))\to g_z(\Diamond\pi)\leq g_z(\Diamond\phi),\]
which is true since $g_z(\Diamond\pi)\to g_z(\Diamond\pi) =1$ and $1\to g_z(\Diamond\pi) = g_z(\Diamond\pi)$. The case in which $\pi: = \lozenge\sigma$ is analogous to the one above, and its proof is omitted. 
\end{proof}


\begin{theorem} \label{thm:completeness}
	The basic multi-type normal $\mathcal{L}_{\mathsc{MT}}$-logic $\mathbf{L}$ is sound and complete w.r.t.~the class of graph-based $\mathbf{A}$-frames. 
\end{theorem}
\begin{proof}
Consider a type-uniform $\mathcal{L}_{\mathrm{MT}}$-sequent $\phi \vdash \psi$ that is not derivable in $\mathbf{L}$. Consider the proper filter $f_{\phi}$  and complement of proper ideal $u_{\psi}$  given by 
\[
f_{\phi}(\chi) = \left\{\begin{array}{ll}
1 &\text{if } \phi \vdash \chi\\
0 &\text{if } \phi \not \vdash \chi
\end{array} \right.
\]
and
\[
u_{\psi}(\chi) = \left\{\begin{array}{ll}
0 &\text{if } \chi \vdash \psi\\
1 &\text{if } \chi \not \vdash \psi.
\end{array} \right.
\]
Then $\bigwedge_{\chi}(f_{\phi}(\chi)\to u_{\psi}(\chi)) = 1$, for else there would have to be a formula $\chi_0 $ such that $f_{\phi}(\chi_0) = 1$ and $u_{\psi}(\chi_0) = 0$, which would mean that $\phi \vdash \chi_0$ and $\chi_0 \vdash \psi$ and hence that $\phi \vdash \psi$, in contradiction with the assumption that $\phi \vdash \psi$ is not derivable. It follows that $(f_{\phi}, u_{\psi})$ is a state (of the appropriate type) in the canonical model $\mathbb{M}$. 
By the Truth Lemma, $\val{\phi}(z) = f_{\phi}(\phi) = 1$, and moreover
\begin{center}
\begin{tabular}{cl}
   & $\descr{\psi}^{[0]}(z)$\\
= & $\bigwedge_{(\alpha, z')\in Z_X}\descr{\psi}(\alpha, z')\to (E(z, z')\to\alpha)$\\
$\leq$ & $\descr{\psi}(0, z)\to (E(z, z)\to 0)$\\
= & $(u_{\psi}(\psi) \to 0)\to (E(z, z)\to 0)$\\
= & $(0 \to 0)\to (1\to 0)$\\
 = & 0,
\end{tabular}
\end{center}
which proves the claim. 
\end{proof}

\end{document}